\documentclass[10pt,a4paper]{amsart} 

\usepackage[utf8]{inputenc}
\usepackage[T1]{fontenc}
\usepackage{amsmath, amssymb}
\usepackage{amsthm}
\usepackage{mathtools}   
\usepackage{bbm}         

\usepackage[dvipsnames]{xcolor}
\usepackage{ifdraft}

\usepackage{xargs}  
\usepackage[colorinlistoftodos,prependcaption,textsize=tiny,obeyFinal]{todonotes}
\newcommandx{\eblq}[2][1=]{\todo[linecolor=red,backgroundcolor=red!25,bordercolor=red,#1]{#2}}

\usepackage{tikz-cd}   

\usepackage{enumitem} 
\setenumerate[0]{label=(\roman*)}   

\usepackage{hyperref}
\hypersetup{
    unicode=true,          
    pdftitle={},    
    pdfauthor={},     
    pdfsubject={},   
    pdfkeywords={}{}, 
    pdfborder={0 0 0}
}

\renewcommand{\phi}{\varphi} 

\newcommand\dash{\nobreakdash-\hspace{0pt}}

\newtheorem{quotedthm}{Theorem}

\newtheorem{thm}{Theorem}[section]
\newtheorem{coro}[thm]{Corollary}
\newtheorem{prop}[thm]{Proposition}
\newtheorem{lemm}[thm]{Lemma}
\newtheorem{conj}[thm]{Conjecture}
\theoremstyle{definition}
\newtheorem{defn}{Definition}
\newtheorem*{rmk}{Remark}

\newcommand{\N}{{\mathbb N}}
\newcommand{\Z}{{\mathbb Z}}
\newcommand{\R}{{\mathbb R}}
\renewcommand{\C}{{\mathbb C}}
\newcommand{\Q}{{\mathbb Q}}
\newcommand{\dd}{{\,\mathrm{d}}}
\newcommand{\Ocal}{\mathcal{O}}

\newcommand{\transp}[1]{\prescript{t}{}{\!#1}} 
\newcommand{\abs}[1]{\lvert#1\rvert}    
\newcommand{\absbig}[1]{\bigl\lvert#1\bigr\rvert}   
\newcommand{\abse}[1]{\left\lvert#1\right\rvert}    
\newcommand{\norm}[1]{\lVert#1\rVert}   
\newcommand{\normbig}[1]{\bigl\lVert#1\bigr\rVert}   

\newcommand{\indic}{\mathbbm{1}}

\DeclareMathOperator{\Aut}{Aut}
\DeclareMathOperator{\Aff}{Aff}
\DeclareMathOperator{\Hom}{Hom}

\DeclareMathOperator{\SL}{SL}
\DeclareMathOperator{\GL}{GL}
\DeclareMathOperator{\Lip}{Lip}
\DeclareMathOperator{\Lie}{Lie}
\DeclareMathOperator{\Heis}{Heis}

\DeclareMathOperator{\Supp}{supp}
\DeclareMathOperator{\tr}{tr}

\DeclareMathOperator{\mes}{m}
\DeclareMathOperator{\NC}{NC}

\newcommand{\pp}{\boxplus}
\newcommand{\mm}{\boxminus}

\newcommand{\Ccal}{\mathcal C}
\newcommand{\Fcal}{\mathcal F}
\newcommand{\Hcal}{\mathcal H}
\newcommand{\Lcal}{\mathcal L}

\newcommand{\Wcal}{\mathcal W}

\newcommand{\Tbb}{\mathbb T}
\newcommand{\Pbb}{\mathbb P}
\newcommand{\Ebb}{\mathbb E}

\newcommand{\acts}{\curvearrowright}

\title{Equidistribution of affine random walks on some nilmanifolds}
\author{Weikun He}
\thanks{The authors are supported by ERC 2020 grant HomDyn (grant no.~833423). W.H. is also supported by KIAS Individual Grant (no. MG080401)}
\address{Korea Institute for Advanced Study, Seoul 02455, Republic of Korea}
\email{heweikun@kias.re.kr}

\author{Tsviqa Lakrec}
\thanks{}
\address{Einstein Institute of Mathematics, The Hebrew University of Jerusalem, Jerusalem 91904, Israel.}
\email{tsviqa@gmail.com}

\author{Elon Lindenstrauss}
\thanks{}
\address{Einstein Institute of Mathematics, The Hebrew University of Jerusalem, Jerusalem 91904, Israel.}
\email{elon@math.huji.ac.il}

\date{\today}
 \dedicatory{Dedicated to the memory of Jean Bourgain}
\begin{document}

\maketitle

\begin{abstract}
We study quantitative equidistribution in law of affine random walks on nilmanifolds, motivated by a result of Bourgain, Furman, Mozes and the third named author on the torus.
Under certain assumptions, we show that a failure to having fast equidistribution is due to a failure on a factor nilmanifold.
Combined with equidistribution results on the torus, this leads to an equidistribution statement on some nilmanifolds such as Heisenberg nilmanifolds. In an appendix we strengthen results of de Saxce and the first named author regarding random walks on the torus by eliminating an assumption on Zariski connectedness of the acting group.
\end{abstract}

\section{Introduction}

In this paper we consider random walks on compact nilmanifolds by automorphisms of the nilmanifolds as well as by affine maps.
Recall that a nilmanifold is a space of the form $X = N/\Lambda$, where $N$ is a connected simply connected nilpotent Lie group, and $\Lambda < N$ is a lattice (which in a nilpotent Lie group is necessarily cocompact; cf.~\cite{Raghunathan}).
An \emph{automorphism} of $X$ is defined to be the homeomorphism of $X$ induced by a Lie group automorphism of $N$ that preserves $\Lambda$; we denote the group of all such automorphisms by $\Aut(X)$.
An \emph{affine transformation} on $X$ is the composition of an automorphism of $X$ by left translation by an element of $N$; the group of affine transformations of $X$, denoted by $\Aff(X)$ is the semidirect product $\Aut(X) \ltimes N$.
The projection $\Aff(X) \to \Aut(X)$ will be denoted by $\theta$.

Given a Borel probability measure $\mu$ on $\Aut(X)$ (or more generally $\Aff(X)$) and a starting point $x \in X$, we can define a random walk by successively applying to $x$ a sequence of elements $g_1$, $g_2$, \dots each $g_i$ chosen i.i.d according to $\mu$. Thus the distribution of the random walk after $n$ steps, i.e. of the random element $g_n...g_1x$ in $X$, is given by $\mu^{*n}*\delta_x$.

\medskip

In this situation,  Bekka and Guivarc'h give a sufficient and necessary condition for the random walk defined by $\mu$ to have a spectral gap on $L^2(X)$:

\begin{quotedthm}[Bekka-Guivarc’h {\cite[Theorem 1]{BekkaGuivarch}}]
\label{thm:BekkaGuivarch}
Let $X=N/\Lambda$ be a nilmanifold and let $H$ be a countable subgroup of $\Aff(X)$. The following are equivalent
\begin{enumerate}
    \item The action of $H$ on $X=N/\Lambda$ has a spectral gap.
    \item The action of $H$ on $T=N/[N,N]\Lambda$ has a spectral gap.
    \item \label{it:nofactorus} There is no non-trivial $H$-invariant factor torus $T'$ of $T$ such that the projection of $H \subset \Aff(X)$ to $\Aut(T')$ is virtually abelian.
\end{enumerate}
\end{quotedthm}
Recall for a torus $T=V/\Delta$ where $V$ is an Euclidean space and $\Delta$ is a lattice in $V$, a factor torus is some $T' = T/S$ where $S$ is a subtorus of $T$ (that is, $S=W/(W\cap\Delta)$ for $W$ a rational linear subspace of $V$ relative to the rational structure defined by $\Delta$)\footnote{Note that finite index quotients are not considered factor tori under this definition.}. 
If $H$ is some subgroup of $\Aff(T)$, then the factor torus $T'$ is said to be $H$-invariant if $W$ is invariant under $\theta(H)$, equivalently if the action of $H$ on $T$ induces an action of $H$ on $T'$ so that the projection map from $T$ to $T'$ is $H$-equivariant.
For a nilmanifold $X = N/\Lambda$, the quotient $N/[N,N]\Lambda$ is a torus, called the \emph{maximal torus factor} of $X$.

When these equivalent conditions in Theorem~\ref{thm:BekkaGuivarch} are satisfied, for all but a set of $x$ of exponentially small measure, the random walk $g_n...g_1 x$ emanating from $x$ is exponentially close to being equidistributed. The purpose of this paper is to understand the random walk $g_n...g_1 x$ starting from \emph{any} $x \in X$.

\medskip
Under certain assumptions on $\mu$, that are substantially stronger than those in Theorem~\ref{thm:BekkaGuivarch}, we show that either the random walk equidistributes, namely, $\mu^{*n}*\delta_x$ converges to the Haar measure $\mes_X$ on $X$ in the weak-$*$ topology or the random walk is trapped in a proper closed set invariant under the group generated by $\Supp(\mu)$.
Furthermore, the equidistribution result is quantitative.
Informally, we show that if the equidistribution is not fast, it is because the random walk is close to a "small" orbit.

\subsection{Quantitative equidistribution}
Similar to the spirit of Theorem~\ref{thm:BekkaGuivarch}, we aim to prove that if a random walk on a nilmanifold does not equidistribute, it is because the projected random walk on a factor torus does not equidistribute.
This leads to the following definitions.

Consider a nilmanifold $X = N/\Lambda$. 
We fix a Riemannian distance on $X$.
For $\alpha \in {(0,1)}$, let $\Ccal^{0,\alpha}(X)$ denote the space of $\alpha$-Hölder continuous functions on $X$, equipped with the norm
\[\norm{f}_{0,\alpha} = \norm{f}_\infty + \sup_{x\neq y \in X} \frac{\abs{f(x) - f(y)}}{d(x,y)^\alpha}.\]
For $\nu$ and $\eta$ Borel measures on $X$, recall that the \emph{$\alpha$-Wasserstein distance} between them is defined by
\[\Wcal_\alpha(\nu,\eta) = \sup_{f \in \Ccal^{0,\alpha}(X) : \norm{f}_{0,\alpha} \leq 1} \abse{\int_X f \dd \nu - \int_X f \dd \eta}.\]



Let $T = V/\Delta$ be a torus of dimension $d$. We choose an identification $\Z^d$ with its group of unitary characters via some isomorphism $a \mapsto \chi_a$. Each closed subgroup $L$ of $T$ is uniquely determined by its dual
\[L^* = \{\,a \in \Z^d \mid L \subset \ker \chi_a \,\}.\]
\begin{defn} A closed subgroup $L$ of a torus $T$ is said to have \emph{height~$\leq h$} if its dual $L^* \subset \Z^d$ can be generated by integer vectors of norm $\leq h$.
\end{defn}

\begin{rmk}
This notion depends on the choice of the isomorphism from $\Z^d$ to the group of unitary characters.
For any torus we encounter in this paper, we assume such choice is implicitly fixed in advance.
\end{rmk}

For the next two definitions, we will denote by $T =N/[N,N]\Lambda$ the maximal torus factor of $X$ and by $\pi \colon X \to T$ the canonical projection.
\begin{defn}
Let $\lambda > 0$, $C > 1$ and $\alpha \in {(0,1]}$ be parameters.
Let $\mu$ be a Borel probability measure on a $\Aut(X)$, and let $\Gamma = \langle \Supp(\mu)\rangle$.
We say that the $\mu$-induced random walk on $X$ satisfies $(C,\lambda, \alpha)$\dash{}\emph{quantitative equidistribution} if the following holds for any integer $m \geq 1$ and any $t \in {(0,\frac{1}{2})}$.
Assume
\[m \geq C \log\frac{1}{t} \quad \text{and} \quad \Wcal_\alpha(\mu^{*m}*\delta_x, \mes_X) > t.\]
Then there exists a point $x' \in X$ such that
\begin{enumerate}
    \item $d(x,x') \leq e^{-\lambda m}$
    \item \label{it:small height} $\pi(\Gamma x')$ lies in a proper closed $\Gamma$-invariant subgroup of $T$ of height~$\leq t^{-C}$.
\end{enumerate}
\end{defn}
\noindent
In the case where $\Gamma$ acts irreducibly on $T$ (that is, $\Gamma$ acts irreducibly on $N/[N,N]$ over $\Q$), the condition \ref{it:small height} can be replaced by
\begin{enumerate}
    \item[(ii')] $\pi(\Gamma x')$ consists of rational points of denominator $\leq t^{-C}$.
\end{enumerate}

\medskip
In the situation of an affine random walk, the definition needs to be adjusted. We fix a left-invariant Riemannian distance on the Lie group $\Aff(X)$.
\begin{defn}
Let $\lambda > 0$, $C > 1$ and $\alpha \in {(0,1]}$ be parameters.
Let $\mu$ be a finitely supported Borel probability measure on $\Aff(X)$.
We say that the $\mu$-induced random walk on $X$ satisfies $(C,\lambda, \alpha)$\dash{}\emph{quantitative equidistribution} if the following holds for any integer $m \geq 1$ and any $t \in {(0,\frac{1}{2})}$.
Assume
\[m \geq C \log\frac{1}{t} \quad \text{and} \quad \Wcal_\alpha(\mu^{*m}*\delta_x, \mes_X) > t.\]
Then there exist a point $x' \in X$ and a closed subgroup $H' \subset \Aff(X)$ such that
\begin{enumerate}
    \item $d(x,x') \leq e^{-\lambda m}$
    \item $d(g,H') \leq e^{-\lambda m}$ for every $g \in \Supp(\mu)$,
    \item \label{it:aff small height} $\pi(H' x') - \pi(x')$ lies in a proper closed $\theta(H')$-invariant subgroup of $T$ of height~$\leq t^{-C}$.
\end{enumerate}
\end{defn}
\noindent
Here, we are thinking of $H'$ as generated by $e^{-\lambda m}$-perturbations of elements of $\Supp(\mu)$. Since $\Aut(X)$ is discrete, the perturbation only happens on the translation part. In particular, we will have $\theta(H') = \theta(H)$ where $H = \langle \Supp(\mu)\rangle$.
Again, if $\theta(H)$ acts irreducibly on $T$, then the condition \ref{it:aff small height} can be replaced by
\begin{enumerate}
    \item[(iii')] $\pi(H' x') - \pi(x')$ consists of rational points of denominator $\leq t^{-C}$.
\end{enumerate}

In both the linear and affine case, 
we say that the $\mu$-induced random walk on $X$ satisfies $(\lambda, \alpha)$\dash{}\emph{quantitative equidistribution} if it satisfies $(C,\lambda,\alpha)$\dash{}quantitative equidistribution for some constant~$C$.

\begin{rmk}
For  $0 < \alpha < \alpha' \leq 1$, we have $\Wcal_{\alpha'}(\nu,\eta) \leq \Wcal_{\alpha}(\nu,\eta)$ for any measures $\nu$ and $\eta$.
Hence, $(C,\lambda,\alpha)$\dash{}quantitative equidistribution implies $(C,\lambda,\alpha')$\dash{}quantitative equidistribution for any $\alpha' \in {(\alpha,1]}$.
\end{rmk}

\subsection{Statement of the main result}
Let $\mu$ be a Borel probability measure on a Lie group $H$.
If $H$ acts on an Euclidean space $Z$ via $\theta_Z \colon H \to \GL(Z)$,
We define the \emph{essential exponential growth rate} of the action on $Z$ to be the quantity
\[\tau_Z(\mu) = \inf_{\kappa  > 0} \limsup_{m \to +\infty} \frac{1}{m} \min \bigl\{\, \log \# A \mid A \subset \Aut(Z) \text{ and } (\theta_Z)_*\mu^{*m}(A) \geq 1 - e^{-\kappa m} \,\bigr\}.\]
Clearly, if $(\theta_Z)_*\mu$ is finitely supported, 
\[\tau_Z(\mu) \leq \lim_{m \to +\infty} \frac{1}{m} \log \left(\# \Supp\left((\theta_Z)_*\mu^{*m}\right)\right).\]



Let $H \acts (X,\mes_X)$ be a probability measure preserving action of $H$ on a compact space $X$. Let $(U_X, L^2(X,\mes_X))$ denote the corresponding Koopman representation. Let $\pi \colon (X,\mes_X) \to (Y, \mes_Y)$ be a factor, i.e. $\pi_*\mes_X = \mes_Y$, $H$ acts on $(Y,\mes_Y)$ and $\pi$ is $H$-equivariant.
By composing with $\pi$, we can embed $L^2(Y,\mes_Y)$ in $L^2(X,\mes_X)$ as an $H$-invariant subspace.
Let $U_{X,Y}$ be the restriction of $U_X$ to the orthogonal complement of $L^2(Y,\mes_Y)$ in $L^2(X,\mes_X)$.
For a Borel probability measure $\mu$ on $H$, define
\[\sigma_{X,Y}(\mu) = -\lim_{m\to +\infty} \frac{1}{m} \log \norm{U_{X,Y}(\mu)^m}.\]

We say a measure  $\mu$ on $\Aff(X)$ has a \emph{finite exponential moment} if there exists $\beta > 0$ such that
\begin{equation}
\label{eq:expmoment}
\int_{\Aff(X)} \Lip_X(g)^\beta \dd \mu(g) < +\infty,
\end{equation}
where for $g \in \Aff(X)$,
\[\Lip_X(g) = \sup_{x,x' \in X,\, x \neq x'} \frac{d(gx,gx')}{d(x,x')}.\]
To keep track of the parameter $\beta$, we say more precisely that $\mu$ has a finite $\beta$-exponential moment.
\begin{thm}
\label{thm:main}
Let $\mu$ be a probability measure on $\Aff(X)$ having a finite $\beta$\dash{}exponential moment for some $\beta > 0$.
Let $\Gamma$ denote the subgroup generated by the support of $\theta_*\mu$.
Assume that there exists a rational $\Gamma$-invariant connected central subgroup $Z \subset N$ such that
\begin{equation}
\label{eq:mainassump}
\tau_Z(\mu) < 2 \sigma_{X,Y}(\mu)
\end{equation}
where $Y = N/(\Lambda Z)$ is the corresponding factor nilmanifold.

If the $\mu$-induced random walk on $Y$ satisfies $(\lambda,\alpha)$\dash{}quantitative equidistribution for some $\lambda > 0$ and $0 < \alpha \leq \min\{1,\beta\}$
then the $\mu$-induced random walk on $X$ satisfies $(\lambda',\alpha)$\dash{}quantitative equidistribution for any $\lambda' \in {(0,\lambda)}$.
\end{thm}

Note that when the equivalent conditions in Theorem~\ref{thm:BekkaGuivarch} hold for $X$ and $H= \langle \Supp(\mu) \rangle$, we have $\sigma_{X,Y}(\mu) > 0$.
In some special situations, for instance, if $\theta_Z(H)$ is a virtually nilpotent group, we have easily $\tau_Z(\mu) = 0$.
Thus, in these situations, Theorem~\ref{thm:main} applies and reduces the problem of quantitative equidistribution on $X$ to whether there is one on $Y$, a nilmanifold of smaller dimension.
The idea is that this will eventually reduce to the case of random walks on a torus, where much more is known.

We believe that a result akin to Theorem~~\ref{thm:main}  should hold more generally with the assumption~\eqref{eq:mainassump} relaxed to $\sigma_{X,Y}(\mu) > 0$, with an appropriate (relative) irreducibility assumption, e.g.

\begin{conj}
Let $\mu$ be a probability measure on $\Aff(X)$ having a finite $\beta$\dash{}exponential moment for some $\beta > 0$.
Let $H$ denote the subgroup generated by $\Supp(\mu)$ and $\Gamma = \theta(H)$.
Assume that there exists a rational $\Gamma$-invariant connected central subgroup $Z \subset N$ with corresponding factor nilmanifold $Y = N/(\Lambda Z)$ so that
\begin{enumerate}
    \item $\sigma_{X,Y}(\mu)>0$. 
    \item for any finite index subgroup $H' <H$, and  any proper $H'$-invariant affine subnilmanifold $X' \subset X$, the projection of $X'$  to $Y$ is a proper affine subnilmanifold of $Y$.
    \item the $\mu$-induced random walk on $Y$ satisfies $(\lambda,\alpha)$\dash{}quantitative equidistribution for some $\lambda > 0$ and $0 < \alpha \leq \min\{1,\beta\}$.
\end{enumerate}
Then the $\mu$-induced random walk on $X$ satisfies $(\lambda',\alpha)$\dash{}quantitative equidistribution for any $\lambda' \in {(0,\lambda)}$.
\end{conj}

\subsection{The case of a torus} 
Previous works~\cite{BFLM,Boyer_Affine,He_schubert,HS,HLL} on the case of a torus have been conducted by Bourgain, Furman, Mozes, Boyer, Saxcé and the authors of the present paper.
The most general result known when this paper is written can be summarized as follows.
Recall that if $\mu$ is a Borel probability measure on $\GL_d(\R)$, the top Lyapunov exponent of $\mu$ is 
\[\lambda_{1,\R^d}(\mu) = \lim_{n \to +\infty} \frac{1}{n} \int \log \norm{g} \dd \mu^{*n}(g).\]

\begin{quotedthm}[\cite{HS,He_schubert,BFLM}; cf. Appendix~\ref{sc:Zconnected}]
\label{thm:torus}
Let $X = \R^d/\Z^d$ for some $d \geq 2$. 
Let $\mu$ be a probability measure on $\Aut(X)$ having finite exponential moment.
Denote by $\Gamma \subset \GL_d(\Z)$ the subgroup generated by $\Supp(\mu)$.

Assume that the action of $\Gamma$ on $\R^d$ is strongly irreducible. 
Then the $\mu$-induced random walk on $X$ satisfies $(\lambda,\alpha)$\dash{}quantitative equidistribution for any $\lambda$ in the range ${(0,\lambda_{1,\R^d}(\mu))}$ and any $\alpha \in {(0,1]}$.
\end{quotedthm}

Note that the assumption that $\Gamma \acts \R^d$ is strongly irreducible implies that $X$ itself is the only non-trivial $\Gamma$-invariant factor torus.
Thus, a proper $\Gamma$\dash{}invariant closed subgroup of $X$ is a finite set of rational points. Its height controls the size of denominators.

\begin{quotedthm}[\cite{HLL,Boyer_Affine}; cf. Appendix~\ref{sc:Zconnected}]
\label{thm:affine torus}
Let $X = \R^d/\Z^d$ for some $d \geq 2$. 
Let $\mu$ be a finitely supported probability measure on $\Aff(X)$.
Denote by $\Gamma \subset \GL_d(\Z)$ the subgroup generated by $\Supp(\theta_*\mu)$.

Assume that the action of $\Gamma$ on $\R^d$ is strongly irreducible. 
Then given $\lambda \in {(0,\lambda_{1,\R^d}(\theta_*\mu))}$ and $\alpha \in {(0,1]}$, there exists $C = C(\theta_*\mu,\lambda,\alpha)$ such that
the $\mu$-induced random walk on $X$ satisfies $(C,\lambda,\alpha)$\dash{}quantitative equidistribution.
\end{quotedthm}
\noindent Note that the constant $C$ depends only on $\theta_*\mu$ and not on the translation part of the elements in $\Supp(\mu)$.

Note that the statements of Theorem~\ref{thm:torus} and Theorem~\ref{thm:affine torus} are somewhat stronger than those in \cite{HS} and \cite{HLL}, in that there is no assumption that the Zariski-closure of $\Gamma$ is connected. In Appendix~\ref{sc:Zconnected} we explain how this assumption can be eliminated.

\subsection{Consequences of the main theorem}
Assume that there is a filtration 
\[1 = Z_0 \subset Z_1 \subset \dots \subset Z_{l-1} \subset Z_l = N\]
of rational closed connected subgroups such that $Z_k/Z_{k-1}$ is central in $N/Z_{k-1}$ for all $k = 1, \dotsc, l$.
Denote $X_k = N/(Z_k\Lambda)$ for $k = 0, \dotsc, l-1$. Thus, we have a tower of nilmanifolds,
\[X = X_0 \to X_1 \to \dots \to X_{l-1} = Z_l/(Z_{l-1}\Lambda)\]
where the last nilmanifold is a torus. 
Theorem~\ref{thm:main} combined with Theorem~\ref{thm:torus} immediately leads to the following statement.
\begin{thm}
\label{thm:tower}
Let $\mu$ be a probability measure on $\Aut(X)$ having a finite exponential moment or a finitely supported probability measure on $\Aff(X)$.
Let $\Gamma$ denote the subgroup of $\Aut(X)$ generated by the support of $\theta_*\mu$.
Assume 
\begin{enumerate}
\item\label{it:cond gamma inv} For all $k = 1, \dotsc, l-1$, $Z_k$ is  $\Gamma$-invariant;
\item\label{it:cond tau sigma} For all $k = 1, \dotsc, l-1$, $\tau_{Z_k/Z_{k-1}}(\mu) < 2 \sigma_{X_{k-1},X_k}(\mu)$;
\item\label{it:cond SI}   The action of $\Gamma $ on $Z_l/Z_{l-1}$ is strongly irreducible.
\end{enumerate}
Then the $\mu$-random walk on $X$ is $(\lambda,\alpha)$\dash{}quantitatively equidistributed for any $\lambda \in {(0,\lambda_{1,Z_l/Z_{l-1}}(\mu))}$ and any $\alpha \in {(0,1]}$.
\end{thm}

Note that unlike in \cite[Theorem 1.3]{HLL} (cf. Theorem~\ref{thm:affine torus}), here the implicit constant $C$ of the $(\lambda,\alpha)$\dash{}quantitative equidistribution does depend on the translation part, though mildly. For more details about the dependence of the implicit constants on the translation part see Lemma~\ref{lm:CSargument}\footnote{Lemma~\ref{lm:CSargument} contains the key inductive step, combined with \cite[Theorem 1.3]{HLL} one can easily see how the constants in Theorem~\ref{thm:tower} depend on the translation parts. As this is not particularly illuminating we do not give an explicit discussion here.}.

From this quantitative statement, i.e.\ Theorem~\ref{thm:tower}, we can deduce easily the following qualitative statement.
\begin{coro}
\label{cor:qualitative main}
Let $\mu$ be either probability measure on $\Aut(X)$ with finite exponential moment or a Borel probability measure on $\Aff(X)$ with finite support. 
Let $\Gamma$ denote the subgroup generated by $\Supp(\theta_*\mu)$,
Let $H$ denote the subgroup generated by $\Supp(\mu)$,

Assume the same assumptions as in Theorem~\ref{thm:tower}.
Then for any $x \in X$, either $\mu^{*m}*\delta_x$ converge to  $\mes_X$ in the weak-$*$ topology or the projection of the orbit $Hx$ to the maximal torus factor is contained in a proper closed $H$-invariant subset. 
\end{coro}

Clearly, the two options in Corollary~\ref{cor:qualitative main} are mutually exclusive.

\medskip

From Corollary~\ref{cor:qualitative main} follow easily the following classification theorem about orbit closures and stationary measures:
if $\mu$ and $\Gamma$ are as in in Theorem~\ref{thm:tower}, then a $\Gamma$-orbit closure is either $X$ or projects to a proper closed $\Gamma$-invariant subset on the maximal torus factor.
Similarly, an ergodic $\mu$-stationary measure on $X$ is either $\mes_X$ or supported on a proper closed invariant subset. 
However, these classification theorems can be deduced from the work of Benoist and Quint~\cite{BenoistQuintIII} and the work of Eskin and the third-named author~\cite{EL}, works that deal with the much more general context of random walks on homogeneous spaces.
For instance \cite[Corollary 1.10]{BenoistQuintIII} states as follows.
In the case of automorphism action, if $\Gamma$ is a finitely generated subgroup of $\Aut(X)$, whose Zariski closure in $\Aut(N)$ is a Zariski connected semisimple subgroup with no compact factor, then every $\Gamma$-orbit closure $\overline{\Gamma x}$ is a finite homogeneous union of affine submanifolds.
If moreover $\mu$ is a probability measure $\Gamma$ whose support generates $\Gamma$, then 
the Ces\`aro mean $\frac{1}{n}\sum_{m=1}^{n}\mu^{*m}*\delta_x$ converges in the weak-$*$ topology to the homogeneous measure probability measure on $\overline{\Gamma x}$ (the measure induced by the Haar measure on the stabilizer of $\overline{\Gamma x}$).
In~\cite{EL}, the requirement on semisimplicity is relaxed. However, there does not seem to be at present a purely ergodic theoretic approach to Corollary~\ref{cor:qualitative main} without the additional Ces\`aro mean.
\subsection{Idea of the proof} To conclude this introduction, let us explain the conceptual ideas behind the proof of the main theorem. Let $f$ be a test function on $X$ which witnesses $\Wcal_\alpha(\mu^{*m}*\delta_x, \mes_X) > t$.
The goal is to construct another witness which has the additional property that it is constant on each fiber of $\pi \colon X \to Y$.
Using Fourier analysis on the fibers (i.e. the $Z$-direction), we may assume without loss of generality that $f$ behaves as a character on the fibers. 
Now sample random elements $g_1$ and $g_2$ in $\Aff(X)$ according to $\mu^{*m'}$ with some $m' \leq m$ and set $f_1 = f\circ g_1$ and $f_2 = f\circ g_2$.
Because $\tau_Z(\mu)$ is small, with large probability, $\theta_Z(g_1) = \theta_Z(g_2)$ so that  $f_1\overline{f_2}$ is constant on each fiber. And because $\sigma_{X,Y}(\mu)$ is large, $f_1\overline{f_2}$ is a witness to $\Wcal_\alpha(\mu^{*(m-m')}*\delta_x, \mes_X) > t^{O(1)}$ also with large probability.

\subsection*{Acknowledgements}
The proof in Appendix~\ref{sc:Zconnected} grew out of a discussion together with Nicolas de Saxcé.
We are grateful to him for sharing his ideas.

This paper is dedicated to the memory of Jean Bourgain, a great man and a profound mathematician, whose deep work laid the framework for much that is done in this paper. In particular, much of what the third named author knows about arithmetic combinatorics he learned from Jean Bourgain. While working on this paper, we have been acutely aware of him being no longer with us --- no doubt if we could have discussed these questions with him we could have gone much further.

\section{Examples}
\label{sc:examples}
This section is devoted to a few concrete examples where our main theorem applies and an example where it does not apply.

\subsection{Heisenberg nilmanifold}
Let $N$ be the $(2d+1)$-dimensional Heisenberg group.
Recall that a Heisenberg group is a two-step connected simply-connected nilpotent Lie group of one dimensional center.
Let $Z$ denote the center of $N$. Note that $[N,N] = Z$ and it is isomorphic to $\R$. 

Let $\Lambda$ be a lattice in $N$ and set $X = N/\Lambda$. 
The maximal factor torus of $X$ is $T = N/[N,N]\Lambda = N/Z\Lambda$.
Let $\mu$ be a Borel probability measure on $\Aff(X)$ with finite support.
Let $H$ denote the subgroup generated by $\Supp(\mu)$ and $\Gamma = \theta(H)$. 
Assume that 
\begin{equation}
\label{eq:HeisSI}
\text{the action of $\Gamma$ on $N/Z$ is strongly irreducible.}
\end{equation}
We claim that the assumptions of Theorem~\ref{thm:tower} are satisfied for the filtration $\{0\} \subset Z \subset N$. 
Indeed, $\{0\} \subset Z \subset N$ is the ascending central series of $N$.
Hence the assumptions on the filtration are satisfied.
It remains to see that $\tau_Z(\mu) < 2 \sigma_{X,T}(\mu)$.
On the one hand, \eqref{eq:HeisSI} implies the condition \ref{it:nofactorus} of Theorem~\ref{thm:BekkaGuivarch}. 
Hence the action of $\Gamma$ on $X$ has a spectral gap, which implies $\sigma_{X,T}(\mu) > 0$ (we remark that the special case of Theorem~\ref{thm:BekkaGuivarch} we use here for Heisenberg nilmanifolds was established by Bekka and Heu in~\cite{BekkaHeu}).
On the other hand, any $\gamma \in \Aut(X)$ preserves both $Z$ and the lattice $Z \cap \Lambda$ in $Z$.
Hence the action of $\Aut(X)$ on $Z$ consists only of $\{ \pm 1\}$.
It follows that $\tau_{Z}(\mu) = 0$, establishing condition \ref{it:cond tau sigma} of Theorems~\ref{thm:tower}, and hence Theorems~\ref{thm:tower} applies to Heisenberg nilmanifolds.

Qualitatively, we can say a little bit more than Corollary~\ref{cor:qualitative main}.
\begin{thm}
Let $X$ be a Heisenberg nilmanifold, $\mu$ a probability measure on $\Aut(X)$ having a finite exponential moment and let $\Gamma$ denote the subgroup generated by $\Supp(\mu)$. Assume the irreducibility condition~\eqref{eq:HeisSI} holds.
Then for every $x \in X$, either $\mu^{*n}*\delta_x$ converges to $\mes_X$ in the weak-$*$ topology or the $\Gamma$-orbit of $x$ is finite.
\end{thm}

\begin{proof}
By the discussion above, Corollary~\ref{cor:qualitative main} applies.
Thus, it is enough to see that if the image of $x$ in $T = N/Z\Lambda$ is rational then the $\Aut(X)$-orbit of $x$ is finite.

By \cite[Theorem 5.1.8]{CorwinGreenleaf}, the $\Q$-span of $\log(\Gamma)$ is a $\Q$-structure of the Lie algebra $\Lie(N)$ of $N$.
We can choose a basis of this $\Q$-structure and identify both $\Lie(N)$ and $N$ with $\R^{2d+ 1} =  \R^{2d} \oplus Z$ so that the projection of $\Lambda$ to $\R^{2d}$ is exactly $\Z^{2d}$.
Then every automorphism $\gamma \in \Aut(X)$ is of the form
\[(y,t) \in \R^{2d} \oplus Z \mapsto (A_\gamma y, \epsilon_\gamma t + L_\gamma y, ) \in  \R^{2d} \oplus Z\]
where $A_\gamma \in \GL_{2d}(\Z)$, $\epsilon_\gamma \in \{ \pm 1\}$ and $L_\gamma \colon \R^{2d} \to \R$ is a linear form.
From \cite[Theorem 5.4.2]{CorwinGreenleaf}, we know that there is an integer $q \in \N$ such that $\Lambda \subset \Z^{2d} \oplus \frac{1}{q}\Z  \subset  \R^{2d} \oplus Z$.
This implies that the linear form $L_\gamma$ must be rational of  denominator $q$.

It follows that if $(y,t) \in \R^{2d} \oplus Z$ with $y$ rational of denominator $q'$, then 
\[\Aut(X)(y,t) \subset  \frac{1}{q'}\Z^{2d} \times \Bigl(\{ \pm t \} + \frac{1}{q q'}\Z\Bigr) 
\subset \R^{2d} \oplus Z.\]
Since the group law in these coordinates is bi-linear with rational structural constants, this allows to conclude that $\Gamma.(y,t)$ is finite.
\end{proof}

\subsection{Heisenberg nilmanifold over number fields}
In the example above, the growth rate $\tau_Z(\mu)$ for the action on the center $Z$ is equal to $0$ because this action is virtually trivial.
In the next example, we have again a 2-step nilpotent group $N$ but the group $\Gamma \subset \Aut(X)$ will have a non-trivial action on the center.

Let $B \colon \C^{2d} \times \C^{2d} \to \C$ be a bilinear form with integral coefficients in the standard basis.
For a commutative ring with unity $R$, define $\Heis_B(R)$ to be the group with underlying set $R^{2d+1} = R^{2d} \times R$ and with the group law $\forall (y,t), (y',t') \in R^{2d} \times R$,
\[(y,t) (y',t') = \bigl(y+y',t +t' + B(y,y')\bigr).\]

Let $K$ be a number field.
Denote by $\Ocal_K$ its ring of integers and by $\Ocal_K^\times$ the group of units.
Let $r_1$ be the number of embeddings of $K$ in $\R$ and $r_2$ the number of conjugate pairs of embeddings of $K$ in $\C$.
Let $\iota \colon K \to \R^{r_1} \times \C^{r_2}$ be the corresponding ring embedding so that $\iota(\Ocal_K)$ is discrete, in fact is a lattice in $\R^{r_1} \times \C^{r_2}$.
This embedding of rings induces an embedding of groups $\Heis_B(\Ocal_K) \to \Heis_B(\R^{r_1} \times \C^{r_2})$, which we denote again by $\iota$.
Let $N = \Heis_B(\R^{r_1} \times \C^{r_2})$ and $\Lambda = \iota(\Heis_B(\Ocal_K))$.
It is easy to check that $\Lambda$ is a lattice in $N$.
Hence $X = N/\Lambda$ is a nilmanifold.

Inside $\Aut(\Lambda)$, we have automorphisms of the form
\[(y,t) \in \Heis(\Ocal_K) \mapsto ( Ay, \epsilon t + Ly) \in \Heis(\Ocal_K)\] 
where $A \subset \GL_{2d}(\Ocal_K)$, $\epsilon \in \Ocal_K^\times$ and $L \in (\Ocal_K^{2d})^*$ such that
\[\forall y,y' \in K^{2d}, B(Ay,Ay') = \epsilon B(y,y').\]
They extend to $\Aut(N)$ via $\iota$.
Denote by $\Gamma_0 \subset \Aut(X)$ the group consisting of such automorphisms.
For example, for $d = 1$ and $B = \det$, $A$ can be any matrix in $\GL_2(\Ocal_K)$ and $\epsilon = \det(A)$, and $\Gamma_0$ is isomorphic to a semi-direct product $\GL_2(\Ocal_K) \ltimes \Ocal_K^2$.

Consider the central subgroup 
\[Z = \{0\} \oplus \R^{r_1} \times \C^{r_2} \subset (\R^{r_1} \times \C^{r_2})^{2d} \times (\R^{r_1} \times \C^{r_2}) = N.\]
Then $\Gamma_0$ preserves $Z$. 
Let $\mu$ be a probability measure on $\Gamma_0$.
Necessarily, $\tau_Z(\mu) = 0$ because $\Gamma_0$ acts on $Z$ via the abelian group $\Ocal_K^\times$, which grows at polynomial rate.
Let $\Gamma$ be the group generated by $\Supp(\mu)$.
The action of $\Gamma$ on $N/Z = (\R^{r_1} \times \C^{r_2})^{2d}$ can be identified with $\Gamma \to \GL_{2d}(\Ocal_K) \xrightarrow{\iota} \GL_{2d}(\R)^{r_1} \times \GL_{2d}(\C)^{r_2}$
If the action of $\Gamma$ on $K^{2d}$ is strongly irreducible over $K$, then the action of $\Gamma$ on $N/Z$ is strongly irreducible over $\Q$.
If moreover $\Gamma$ is not virtually abelian, then the condition~\ref{it:nofactorus} of Theorem~\ref{thm:BekkaGuivarch} is satisfied.
We conclude that
\[\tau_Z(\mu) = 0 < 2\sigma_{X,T}(\mu)\]
where $T = N/(\Lambda Z)$. Hence Theorem~\ref{thm:main} applies.
However, Theorem~\ref{thm:torus} does \textbf{not} apply to the induced random walk on $T$, as the action of $\Gamma$ on $N/Z$ is is not irreducible over $\R$ unless $r_1+r_2 =1$ (it is strongly irreducible over $\Q$ unless $K$ is a totally complex extension of a totally real field, c.f.\ e.g.\ \cite[\S2]{Zhiren}).
However, it is conjectured that a quantitative equidistribution holds for such random walks on $T$, at least under the assumption that the projection of $\Gamma$ to $\GL_{2d}(\R)^{r_1} \times \GL_{2d}(\C)^{r_2}$ has semisimple Zariski-closure with no compact factor.

\subsection{A non-semisimple group of toral automorphisms}
\label{ss:GRvsSG}
In both examples above, the growth rates of the action on the fibers are all zero. Now we give an example where we have a positive growth rate while our result still applies.

Consider $X = \Tbb^{2d} = \R^{2d}/\Z^{2d}$ with $d \geq 2$. 
Let $A$ and $D$ be independent random elements in $\SL_d(\Z)$.
Denote by $\eta$ the law of $A$ and $\nu$ that of $D$.
Let $I_d$ denote the $d \times d$ identity matrix. 
Let $\mu$ be the law of the random bloc-triangular matrix 
\[\left(\begin{array}{c|c}
A & I_d\\
\hline
0 & D
\end{array}\right).
\]
Let $Z = \R^{d} \oplus \{0\} \subset \R^{2d}$ and $Y = \R^{2d}/\R^{d} \oplus \Z^d$.
The filtration $\{0\} \subset Z \subset \R^{2d}$ is preserved by $\Gamma$, the group generated by the support of $\mu$.

\begin{prop}
\label{pr:positivetau}
In the above setting, given the measure $\eta$, there is some $\nu$ such that Theorem~\ref{thm:tower} can be applied to $\mu$ and the filtration $\{0\} \subset Z \subset \R^{2d}$.
\end{prop}
As a consequence, we can say the following about orbit closures under the action of $\Gamma$, the group generated by the support of the constructed $\mu$.
For every $x \in \Tbb^{2d}$, either $\Gamma x$ is dense or $\Gamma x$ is contained in a finite union of affine subtori parallel to $\R^{d}/\Z^d \oplus \{0\}$.
For properly chosen $\eta$, the group $\Gamma$ will \textbf{not} have semisimple Zariski closure.
Thus, the work of Benoist-Quint~\cite{BenoistQuintIII} does not apply to such group.
Neither does the work of Guivarc'h-Starkov~\cite{GS} nor that of Muchnik~\cite{Muchnik} (though stationary measures even in this case are analyzed by Eskin and the third named author in~\cite{EL}).

To show the proposition, we need the following lemma to control $\sigma_{X,Y}(\mu)$.
\begin{lemm}
\label{lm:UXY}
In the setting above, denote by $(U_Y, L^2(Y,\mes_Y))$ the Koopman representation associated to the action of $\Aut(Y)$ on $Y$ and by $U_{Y,0}$ the restriction of $U_Y$ to the subspace of mean zero functions. 
Then we have
\[\norm{U_{X,Y}(\mu)^2} \leq \sqrt{3} \norm{U_{Y,0}(\nu)}.\]
\end{lemm}

\begin{proof}
Let $F \colon L^2(X,m_X) \to \ell^2(\Z^{2d})$ denote the isometry given by the Fourier transform.
Under this isometry, $U_{X,Y}$ is conjugated to a unitary representation $T$ of $\Gamma$ on $\ell^2((\Z^d \setminus \{0\}) \times \Z^d)$.
Explicitly, 
let $\phi \in \ell^2((\Z^d \setminus \{0\}) \times \Z^d)$.
Then for all $(a,b) \in (\Z^d \setminus \{0\}) \times \Z^d$,
\[(T(\mu)\phi)(a,b) = \int_{\Gamma} \phi(\transp{g}(a,b)) \dd \mu(g) = \Ebb\bigl[\phi(\transp{A}a, a + \transp{D}b)\bigr].\]

Let $P_0$ be the orthogonal projection $\ell^2((\Z^d \setminus \{0\}) \times \Z^d) \to \ell^2((\Z^d \setminus \{0\})\times \{0\})$. 
Concretely, for $\phi \in \ell^2((\Z^d \setminus \{0\}) \times \Z^d)$ and $(a,b) \in (\Z^d \setminus \{0\}) \times \Z^d$,
\[(P_0\phi)(a,b) = \delta_0(b) \phi(a,0).\]
Then $P_0T(\mu)P_0 = 0$ because
\[(P_0T(\mu)P_0\phi)(a,b) = \delta_0(b)\Ebb\bigl[\delta_0(a) \phi(\transp{A}a, 0)\bigr] = 0.\]

Hence, taking the square of the equality $T(\mu) = P_0 T(\mu) + (1 - P_0) T(\mu)$, we see,
\[\norm{T(\mu)^2} \leq 3\norm{(1 - P_0)T(\mu)}.\]
To conclude, it suffices to show
\begin{equation}
\label{eq:1P0T}
\norm{(1 - P_0)T(\mu)} \leq \norm{U_{Y,0}(\nu)}.
\end{equation}

We first show the inequality in the case where $A$ is almost surely some fixed matrix $g \in \SL_d(\Z)$.
Consider, for $a \in \Z^d\setminus \{0\}$, the subspace 
\[\Hcal_a = \ell^2(\{a\} \times \Z^d) \subset \ell^2((\Z^d \setminus \{0\}) \times \Z^d).\]
Let $Q_a$ denote the orthogonal projection onto $\Hcal_a$.
Observe that
\[\forall a \in \Z^d\setminus\{0\}, T(\mu) \Hcal_a \subset \Hcal_{\transp{g}^{-1}a}.\]
Moreover $P_0$ preserves the subspaces $\Hcal_a$.
Hence, for any $\phi \in \ell^2((\Z^d \setminus \{0\}) \times \Z^d)$, the vectors $(1-P_0)T(\mu)Q_a \phi$, $a \in \Z^d \setminus \{0\}$, are all orthogonal to each other. 
Thus
\begin{align*}
\norm{(1-P_0)T(\mu)\phi}^2 &= \sum_{a \in \Z^d \setminus \{0\} } \norm{(1-P_0)T(\mu)Q_a\phi}^2\\
&\leq \sum_{a \in \Z^d \setminus \{0\}} \norm{(1-P_0)T(\mu)Q_a}^2 \norm{Q_a\phi}^2\\
&\leq \Bigl(\sup_{a \in \Z^d \setminus \{0\}} \norm{(1-P_0)T(\mu)Q_a}^2 \Bigr) \norm{\phi}^2.
\end{align*}
By identifying $\Hcal_a$ with $\ell^2(\Z^d)$ in the obvious way, we see that $\norm{(1-P_0)T(\mu)Q_a} = \norm{V_a}$ where $V_a \colon \ell^2(\Z^d) \to \ell^2(\Z^d \setminus \{0\})$ is the operator defined by
\[\forall \psi \in \ell^2(\Z^d),\, \forall b \in \Z^d \setminus \{0\}, \quad (V_a\psi)(b) = \Ebb\bigl[\psi(\transp{g}^{-1}a + \transp{D}b)\bigr].\]
Let $W_a \colon \ell^2(\Z^d) \to \ell^2(\Z^d)$ be the isometry induced by translating the index by $\transp{g}^{-1}a$, so that $V_a = V_0W_a$. 
But $V_0$ is conjugated to $U_{Y,0}(\nu)$ via the Fourier transform. Hence
\[\sup_{a \in \Z^d \setminus \{0\}} \norm{(1-P_0)T(\mu)Q_a} \leq \sup_{a \in \Z^d \setminus \{0\}} \norm{V_a} \leq \norm{V_0} = \norm{U_{Y,0}(\nu)}.\]
This shows \eqref{eq:1P0T} for the special case where $A$ is almost surely constant.

Using the independence between $A$ and $D$. We can write
\[\mu = \int_{\SL_d(\Z)} \mu_g \dd \eta(g),\]
with  $\mu_g$ being the law of the random matrix 
\[\left(\begin{array}{c|c}
g & I_d\\
\hline
0 & D
\end{array}\right).
\]
Then $\norm{T(\mu)} \leq \int_{\SL_d(\Z)} \norm{T(\mu_g)} \dd \eta(g)$ proves \eqref{eq:1P0T}.
\end{proof}

\begin{proof}[Proof of Proposition~\ref{pr:positivetau}]
Once $\eta$ is chosen. The action of $\Gamma$ on $Z$ is determined. Hence $\tau_Z(\mu)$ is determined. 

Let $\nu_0$ be a symmetric probability measure on $\SL_d(\Z)$ whose support generates a Zariski-dense subgroup.
Then by a result of Furman and Shalom~\cite[Theorem 6.5]{FurmanShalom} (which is a special case of Theorem~\ref{thm:BekkaGuivarch}), 
\[\norm{U_{Y,0}(\nu_0)} < 1.\]

Let $\nu = \nu_0^{*k}$ where $k$ is an integer.
By choosing $k$ large enough we can make $\norm{U_{Y,0}(\nu_0)}$ arbitrarily small and hence $\sigma_{X,Y}(\mu)$ arbitrarily large by Lemma~\ref{lm:UXY}. 
This ensures that
\[\tau_Z(\mu) < 2\sigma_{X,Y}(\mu).\]
At the same time, the support of $\nu$ generates a Zariski dense subgroup $\Gamma$ in $\SL_d$. In particular the action of $\Gamma$ on $\R^{2d}/Z$ is strongly irreducible.
This is why Theorem~\ref{thm:tower} can be applied.
\end{proof}

\subsection{A non-example}
Let $N$ be the connected and simply-connected nilpotent Lie group whose Lie algebra is the free $2$-step nilpotent Lie algebra on 3 generators.
It can be realised as $N = \R^3 \oplus \R^3$ with the group law being
\[(x,y)(x',y') = (x + x', y + y' + x \wedge x'),\quad \text{for all } x,x',y,'y \in \R^3,\]
where $\wedge$ denotes the usual cross product on $\R^3$.
As explained in \cite[Example 35]{BekkaGuivarch}, the automorphism group $\Aut(N)$ of $N$ is isomorphic to the subgroup of $\GL_6(\R)$ of matrices $g_{A,B}$ of the form
\[g_{A,B} = 
\left(\begin{array}{c|c}
A & 0\\
\hline
B & \det(A) (A^{\tr})^{-1}
\end{array}\right),\]
with $A \in \GL_3(\R)$ and $B$ any $3 \times 3$ matrix with real coefficients.
Here $\Aut(N)$ acts on the center $Z$ of $N$ via $\theta_Z \colon g_{A,B} \mapsto \det(A) (A^{\tr})^{-1}$ and acts on $N/Z$ via $\theta_{N/Z} \colon g_{A,B} \mapsto A$.

Let $\Lambda$ be any lattice in $N$ and set $X = N/\Lambda$.
Let $\mu$ be a probability measure on $\Aut(X)$ and $\Gamma$ the group generated by its support.
Denote moreover $Y = N/(\Lambda Z)$.
In order to apply Theorem~\ref{thm:main} to the factor map $X \to Y$, we need 
\begin{enumerate}
\item $\tau_Z(\mu)$ to be small; informally, that is $\theta_Z(\Gamma)$ is a small group;
\item $\sigma_{X,Y}(\mu)$ to be large; in view of Theorem~\ref{thm:BekkaGuivarch}, this requires  $\theta_{N/Z}(\Gamma)$ to be a large group (not virtually amenable by \cite[Theorem 1 and Theorem 5]{BekkaGuivarch});
\end{enumerate}
But $\theta_Z(\Gamma)$, isomorphic to $\theta_{N/Z}(\Gamma)$, cannot be small and large at the same time.
This is why, very likely, Theorem~\ref{thm:main} does not apply to such random walks.
However, we still expect the conclusion of Theorem~\ref{thm:main} to hold, provided that $\theta_{N/Z}(\Gamma)$ is a large group (e.g. Zariski dense in $\SL_3(\R)$).

\section{The setup}
\label{sc:setup}
Throughout this paper, $X = N/\Lambda$ denotes a nilmanifold.
As recalled in the introduction, this means that $N$ is a connected simply connected nilpotent Lie group, $\Lambda \subset N$ is a lattice, which is necessarily cocompact (\cite[Theorem 2.1]{Raghunathan}).
Recall that the $\Q$-span of $\log(\Lambda)$ defines a $\Q$-structure on $\Lie(N)$, the Lie algebra of $N$.
A connected closed subgroup of $N$ is said to be rational if its Lie algebra is rational in $\Lie(N)$ with respect to this $\Q$-structure.
For a connected closed subgroup $M \subset N$ to be rational it is necessary and sufficient that $M \cap \Lambda$ is a lattice in $M$. For these, see \cite[\S 5.1]{CorwinGreenleaf}.
    
Denote by $\Aut(X) = \Aut(N/\Lambda)$ denote the group of continuous automorphisms of $N$ preserving $\Lambda$.
Let $\Aff(X) = \Aff(N/\Lambda) = \Aut(X) \ltimes N$ denote the group of (invertible) affine transformations of $X$.
More precisely for $\gamma \in \Aut(X)$ and $n \in N$, let $(\gamma,n) \in \Aff(X)$ denote the map $X \to X$, $x\Lambda \mapsto n\gamma(x)\Lambda$.
Denote by $\theta \colon \Aff(X) \to \Aut(X)$ the projection to the automorphism part, that is, $\theta(\gamma,n) = \gamma$ for all $(\gamma,n) \in \Aff(X)$.

Moreover, we will identify an automorphism $\gamma \in \Aut(X)$ with $(\gamma,1_N) \in \Aff(X)$ and an element $n \in N$ with the left translation $(1,n) \in \Aff(X)$.
With this notation we have, for all $\gamma \in \Aut(X)$ and all $n \in N$, $\gamma n \gamma^{-1} = \gamma(n)$.
If $g \in \Aff(X)$ and $n \in N$ is central, then $gng^{-1} = \theta(g)(n)$.

Let $\mes_X$ denote the normalised $N$-invariant measure on $X$ induced by the Haar measure of $N$.
The action $\Aff(X) \acts X$ preserves $\mes_X$.
Let $(U,L^2(X,\mes_X))$ denote the associated Koopman representation.
That is, for $g \in \Aff(X)$, $U(g)$ is the unitary operator on $L^2(X,\mes_X)$ defined by
\[\text{for all $f \in L^2(X,\mes_X)$ and almost all $x \in X$, } U(g)f (x) = f(g^{-1}(x)).\]
Let also $U^*(g) = U(g)^* = U(g^{-1})$.
By an abuse of notation, we let $U(g)$ and $U^*(g)$ denote also the operators from $\Ccal^0(X)$, the space of continuous functions, to itself defined in the obvious way.

Let $\mu$ be a Borel measure on $\Aff(X)$.
We set $U(\mu) = \int U(g) \dd \mu(g)$ and $U^*(\mu) = \int U^*(g) \dd \mu(g)$.
For any integer $m \geq 0$, any Borel measure $\eta$ on $X$ and any continuous function $f \in \Ccal^0(X)$, we have
\[ \int_X f \dd \mu^{*m}*\eta = \int_X U^*(\mu)^m f \dd \eta.\]

\subsection{Hölder functions}
We fix a Riemannian metric on $X$ and let $d \colon X \times X \to {[0,+\infty)}$ denote the associated distance function.
Let $\alpha \in {(0,1]}$. 
Denote by $\Ccal^{0,\alpha}(X)$ the set of $\alpha$-Hölder continuous functions from $X$ to $\C$. Endow it with the norm
\[\norm{f}_{0,\alpha} = \norm{f}_\infty + \omega_\alpha(f)\]
where
\[\omega_\alpha(f) = \sup_{x \neq y \in X} \frac{\abs{f(x) - f(y)}}{d(x,y)^\alpha}.\]


For $g \in \Aff(X)$, define
\[\Lip_X(g) = \sup_{x,x' \in X,\, x \neq x'} \frac{d(gx,gx')}{d(x,x')}.\]
This quantity is finite since $g$ is of class $\mathcal{C}^\infty$ and $d$ is a Riemannian distance.
It is greater or equal to $1$ since $X$ is compact.
Moreover,  $\Lip_X \colon \Aff(X) \to {[1,+\infty)}$ is continuous and submultiplicative, i.e. for all $g,h \in G$,
\begin{equation}
\label{eq:product_Lip}
\Lip_X(gh) \leq \Lip_X(g) \Lip_X(h).
\end{equation}

It is straight forward to check that if $g \in \Aff(X)$ and $f \in \Ccal^{0,\alpha}(X)$, then $U^*(g)f$ is still $\alpha$-Hölder continuous and
\[\norm{U^*(g)f}_{0,\alpha} \leq \Lip_X(g)^\alpha \norm{f}_{0,\alpha}.\] 

Remark also that for $f_1, f_2 \in \Ccal^{0,\alpha}(X)$, then $f_1f_2 \in \Ccal^{0,\alpha}(X)$ and 
\begin{equation}
\label{eq:product_Holder}
\norm{f_1f_2}_{0,\alpha} \leq  \norm{f_1}_{0,\alpha} \norm{f_2}_{0,\alpha}.
\end{equation}

\section{The main argument}
\label{sc:induction}
As in the statement of Theorem~\ref{thm:main}, let $\mu$ be Borel measure on $\Aff(X)$ having a finite exponential moment.
Let $\Gamma \subset \Aut(X)$ denote the subgroup generated by the support of $\theta_*\mu$.
Let $Z \subset N$ be a $\Gamma$-invariant rational connected closed central subgroup.
Then $Y = N/(\Lambda Z)$ is a nilmanifold and we have a $\Gamma \ltimes N$-equivariant factor map $\pi \colon X \to Y$.
Let $\mes_Y$ denote the $N$-invariant probability measure on $Y$ induced by the Haar measure of $N$.
We defined two quantities $\tau_Z(\mu)$ and  $\sigma_{X,Y}(\mu)$ in the introduction.
This section is dedicated to the proof of the following proposition.
\begin{prop}
\label{pr:XtoY}
Assume that $\mu$ has a finite $\beta$\dash{}exponential moment.
Assume 
\[\tau_Z(\mu) < 2\sigma_{X,Y}(\mu).\]
Then given $0 < \alpha \leq \min\{1,\beta\}$, there exists a constant $C \geq 2$ such that the following holds.

For any Borel probability measure $\eta$ on $X$, any $t \in {(0,1/2)}$ and any $m \geq  C\log \frac{1}{t}$,
if
\[\Wcal_\alpha(\mu^{*m}*\eta, \mes_X) \geq t\] then 
\[\Wcal_\alpha(\pi_*\eta, \mes_Y) \geq e^{-Cm}.\]
\end{prop}
In other words, if there is $f \in \Ccal^{0,\alpha}(X)$ satisfying
\[\abse{\int_X f \dd \mu^{*m}*\eta - \int_X f \dd \mes_X} > t\norm{f}_{0,\alpha},\]
then there exists $\phi \in \Ccal^{0,\alpha}(Y)$ such that 
\[\abse{\int_X \phi \dd \pi_* \eta - \int_X \phi \dd \mes_Y} > e^{-Cm} \norm{\phi}_{0,\alpha}.\]

\subsection{Principal torus bundle} 
Let $S = Z/(Z \cap \Lambda)$. 
Let $d = \dim Z$.
Then $S$ is a torus of dimension $d$.
Note that $\pi$ is a fiber bundle of fiber $S$.
Moreover, it is a principal bundle: since $Z$ is contained in the center of $N$, the action of $Z$ by left translation on $X$ factors through $S$.

By choosing a basis in $Z\cap \Lambda$, we fix an isomorphism between $\Z^d$ and the group $\Hom(S,S^1)$ of unitary characters of $S$.
Denote the isomorphism as $a \mapsto \chi_a$, $a \in \Z^d$.
The Koopman representation $U$ restricted to $Z$ factors through $S$. Hence we can decompose $L^2(X,\mes_X)$ into a Hilbert sum of characteristic subspaces
\begin{equation}
\label{eq:decompHa}
L^2(X,\mes_X) = \sum_{a \in \Z^d} \Hcal_a
\end{equation}
where for $a \in \Z^d$,
\[\mathcal{H}_a = \{\, f \in L^2(X,\mes_X) \mid \forall z \in Z,\, U(z)f = \chi_a(z)f\,\}.\]
Here we identified $\chi_a$ with its lift as character of $Z$.
For $a = 0$, $\Hcal_0$ is the subspace of functions that are constant on each fiber of $\pi$. 
Since $\pi_*\mes_X = \mes_Y$, we have the isometry
\[\Hcal_0 = L^2(Y,\mes_Y).\]
Thus, the Hilbert space of the representation $U_{X,Y}$ is precisely $\sum_{a \in \Z^d \setminus \{0 \}} \Hcal_a$.

Since for all $g \in \Aff(X)$ and $z \in Z$, $zg = g \theta(g)^{-1}(z)$, we have 
\[\forall g \in \Aff(X),\, \forall a \in \Z^d,\quad U(g)\Hcal_a = \Hcal_{\theta(g)\cdot a},\]
where $\gamma \cdot a \in \Z^d$ is such that $\chi_{\gamma \cdot a} = \chi_a \circ \gamma^{-1}$ for $\gamma \in \Gamma$. This defines an action of $\Gamma$ on $\Z^d$.
Note that $\Gamma$ acts via some homomorphism $\Gamma \to \GL_d(\Z)$.

\subsection{Fourier transform}
For continuous functions, the decomposition~\eqref{eq:decompHa} can be made more explicit using Fourier transforms. 
The aim here is to prove the following lemma using Fourier transforms.
\begin{lemm}
\label{lm:finda0}
Given $\alpha \in {(0,1]}$, there is a constant $C$ depending on $\alpha$ such that the following holds. 
If a measure $\eta$ on $X$, a function $f \in \Ccal^{0,\alpha}(X)$ and $t \in {(0,1/2)}$ satisfy
\[\abse{\int_X f \dd \eta - \int_X f \dd \mes_X} \geq t \norm{f}_{0,\alpha}\]
then there exist $a_0 \in \Z^d$ with $\norm{a_0} \leq t^{-C}$ and $f_0 \in \Ccal^{0,\alpha}(X) \cap \Hcal_{a_0}$ such that
\[\abse{\int_X f_0 \dd \eta - \int_X f_0 \dd \mes_X} \geq t^C\norm{f_0}_{0,\alpha}.\]
\end{lemm}
Specialising this lemma to the case where $X$ is a torus and $Y$ is a point, we can recover Lemma 4.5 in Boyer~\cite{Boyer}. Our proof is slightly shorter.

Let $\mes_S$ denote the normalised Haar measure on $S$. 
For $a \in \Z^d$, define for any $f \in \Ccal^0(X)$,
\[F_a f (x) = \int_S \chi_a(z) (U(z)f)(x) \dd \mes_S(z).\]
It is readily check that $F_a f \in \Hcal_a$. It preserves $\Ccal^{0,\alpha}(X)$ for any  $\alpha \in {(0,1]}$.
Moreover, since $\Lip_X$ is continuous and $S$ is compact, we have, uniformly in $a$,
\begin{equation}
\label{eq:Fa0alpha}
\forall f \in \Ccal^{0,\alpha}(X),\quad \norm{F_a f}_{0,\alpha} \ll \norm{f}_{0,\alpha}.
\end{equation}
Define also the Féjer kernel: for $N \in \N$,
\[\Fcal_N = \sum_{(a_i) \in [-N , N]^d} \, \Biggl(\prod_{i=1}^d \Bigl(1 - \frac{\abs{a_i}}{N}\Bigr) \Biggr) F_{(a_i)}.\]
\begin{lemm}
\label{lm:Fejer}
Let $\alpha \in {(0,1)}$ and $N \in \N$. For any $f \in \Ccal^{0,\alpha}(X)$,
\begin{equation}
\label{eq:Fejer}
\norm{\Fcal_N f - f}_{\infty} \ll N^{-\alpha} \norm{f}_{0,\alpha}.
\end{equation}
\end{lemm}
Here the implied constant depend on the choice of the basis on $Z \cap \Lambda$.
\begin{proof} 
For $t \in \Tbb = \R/\Z$, write $e(t) = e^{2\pi i t}$. 
While defining $\chi_a$, we had chosen a basis of the lattice $Z\cap \Lambda$. 
This choice induces an isomorphism $\phi \colon \Tbb^d \to S$ so that for all $t = (t_1,\dotsc, t_d) \in \Tbb^d$ and all $a = (a_1, \dotsc, a_d) \in \Z^d$,
\[\chi_a(\phi(t)) = \prod_{i=1}^d e(a_it_i).\]

For $N \geq 1$, denote by $K_N \colon \Tbb \to \R$ the $N$-th Féjer kernel on the circle, i.e.
\[\forall t\in \Tbb,\; K_N(s) = \sum_{a = -N + 1}^{N-1} \Bigl(1 - \frac{\abs{a}}{N}\Bigr) e(at) = \frac{1}{N} \Bigl( \frac{\sin(N\pi t)}{\sin(\pi t)} \Bigr)^2.\]

Let $f \in \Ccal^{0,\alpha}(X)$.
It follows from the definition that for all $x \in X$,
\[\Fcal_N f(x) = \int_{\Tbb^d} \biggl(\prod_{i=1}^d K_N(t_i)\biggr) f(\phi(t_1,\dotsc,t_d)^{-1}x) \dd  t_1 \dotsm \dd t_d.\]

We fix a Riemannian distance $d_{\Tbb^d}$  on $\Tbb^d$.
Since $\phi$ is smooth and both $\Tbb^d$ and $X$ are compact,
\[\forall t \in \Tbb^d,\; \forall x \in X,\quad d(\phi(t)^{-1}x,x) \ll d_{\Tbb^d}(t,0)\]
where the implied constant depends on the choice of $d_{\Tbb^d}$.
It follows that  
\begin{align*}
\abs{f(x) - \Fcal_N f(x)} &= \int_{\Tbb^d}  \biggl(\prod_{i=1}^d K_N(t_i)\biggr) \abs{f(\phi(t_1,\dotsc,t_d)^{-1}x) - f(x)} \dd t_1 \dotsm \dd t_d\\
&\leq  \int_{\Tbb^d} \prod_{i=1}^d K_N(t_i) \norm{f}_{0,\alpha} d(\phi(t_1,\dotsc,t_d)^{-1}x,x)^\alpha \dd t_1 \dotsm \dd t_d\\
&\ll \norm{f}_{0,\alpha} \int_{\Tbb^d} \biggl(\prod_{i=1}^d K_N(t_i)\biggr) d_{\Tbb^d} ((t_1,\dotsc,t_d),0)^\alpha \dd t_1 \dotsm \dd t_d
\end{align*}
Note that for $(t_1,\dotsc,t_d) \in {[-\frac{1}{2},\frac{1}{2}]}^d$, we have $d_{\Tbb^d} ((t_1,\dotsc,t_d),0)^\alpha \ll t_1^\alpha + \dots + t_d^\alpha$.
Hence
\begin{equation*}
\int_{\Tbb^d} \biggl(\prod_{i=1}^d K_N(t_i)\biggr) d_{\Tbb^d} ((t_1,\dotsc,t_d),0)^\alpha \dd t_1 \dotsm \dd t_d \ll \int_{-\frac{1}{2}}^{\frac{1}{2}} K_N(t) \abs{t}^\alpha \dd  t.
\end{equation*}
The last quantity is bounded by $N^{-\alpha}$, by \cite[Lemma 1.6.4]{ButzerNessel}.
\end{proof}

\begin{proof}[Proof of Lemma~\ref{lm:finda0}]
We first prove the lemma for $\alpha \in {(0,1)}$. 
Let $C$ denote the implied constant in \eqref{eq:Fejer}. Pick an integer $N$ such that 
\[\frac{t}{8C}\leq N^{-\alpha} \leq \frac{t}{4C}.\]
By Lemma~\ref{lm:Fejer}, we get
\[\norm{\Fcal_N f - f}_{\infty} \leq \frac{t}{4} \norm{f}_{0,\alpha}.\]
Combined with the assumption, this gives
\[\abse{\int_X \Fcal_N f \dd (\eta - \mes_X)} \geq \frac{t}{2}\norm{f}_{0,\alpha}.\]
Then by the definition of the Féjer kernel,
\[\frac{t}{2}\norm{f}_{0,\alpha} \leq \sum_{a \in {[-N,N]}^d} \abse{\int_X F_a f \dd (\eta - \mes_X)}\]
Hence there exists $a \in {[-N,N]}^d$ such that
\[\abse{\int_X F_a f \dd (\eta - \mes_X)} \geq \frac{t}{2(2N+1)^d}\norm{f}_{0,\alpha} \gg t^{1+\frac{d}{\alpha}}\norm{f}_{0,\alpha}.\]
Thus, on account of \eqref{eq:Fa0alpha}, $f_0 = F_a f$ satisfies the required properties.

If $\alpha = 1$, then \eqref{eq:Fejer} in Lemma~\ref{lm:Fejer} becomes (cf. \cite[Lemma 1.6.4]{ButzerNessel})
\[\norm{\Fcal_N f - f}_{\infty} \ll \frac{\log N}{N} \norm{f}_{0,1}.\]
The rest of the proof is similar.
\end{proof}

\subsection{Essential growth rate}
Recall the definition of the quantity $\tau_Z(\mu)$ from the introduction.
Consider a Borel probability measure $\mu$ on $\Aut(Z)$ where $Z$ is a connected simply-connected abelian Lie group. 
For $\kappa > 0$, let 
\[\tau_Z(\mu,\kappa) = \limsup_{m \to +\infty} \frac{1}{m} \log \min\{\,\# A \mid A \subset \Aut(Z) \text{ with } \mu^{*m}(A) \geq 1 - e^{-\kappa m}\,\}.\]
This quantity is non-decreasing in $\kappa$. 
Let 
\[\tau_Z(\mu) = \lim_{\kappa \to 0} \tau_Z(\mu,\kappa).\]
We define similarly $\tau_Z(\mu)$ if, more generally, $\mu$ is a measure on a group which acts measurably on $Z$ by automorphisms. 

Under an exponential moment assumption, this quantity is finite. Moreover it can be bounded in terms of the top Lyapunov exponent of $\mu$. 
\begin{lemm}
Assume that the support of $\mu$ preserves a lattice of $Z$. Assume that $\mu$ has a finite exponential moment. Then
\[\tau_Z(\mu) \leq (d^2 - 1) \lambda_{1,Z}(\mu)\]
where $d = \dim Z$ and $\lambda_{1,Z}(\mu)$ denote the top Lyapunov exponent of the linear random walk defined by $\mu$ on $Z$.
\end{lemm}
\begin{proof}
Without loss of generality, we assume $Z = \R^d$ and that $\Supp(\mu)$ preserves the lattice $\Z^d$.
By the large deviation estimate (Theorem~\ref{thm:LD} proved in the Appendix), for any $\omega > 0$ there is some $\kappa > 0$ such that for all $m$ sufficiently large
\[\mu^{*m}(B(0,e^{(\lambda_{1,Z}(\mu)+\omega)m})) \geq 1 - e^{-\kappa m}.\]
By taking $A = \GL_d(\Z) \cap B(0,e^{(\lambda_{1,Z}(\mu)+\omega)m})$, we get 
\[\tau_Z(\mu) \leq \tau_Z(\mu,\kappa) \leq (d^2 - 1)(\lambda_{1,Z}(\mu)+\omega).\]
We obtain the desired inequality by letting $\omega \to 0$.
\end{proof}

\subsection{The Cauchy-Schwarz argument}
The heart of the proof of Proposition~\ref{pr:XtoY} is a use of the Cauchy-Schwarz inequality.
Let
\[
C_\beta = \int_{\Aff(X)} \Lip_X(g)^\beta \dd \mu(g).
\]

\begin{lemm}
\label{lm:CSargument}
Assume that $\mu$ has a finite $\beta$\dash{}exponential moment (i.e.\ that $C_\beta<\infty$), and that 
\[\tau_Z(\mu) < 2\sigma_{X,Y}(\mu).\]
Then for every $0 < \alpha \leq \min\{1,\beta\}$, there exists a constant $m_0$ depending on $\mu$, and $C$ depending on $\theta_*\mu$ and $ 2\sigma_{X,Y}(\mu) - \tau_Z(\mu)$, such that the following holds.
Let $t \in {(0,1/2)}$ and $f \in \Ccal^{0,\alpha}(X) \cap \Hcal_{a_0}$ with $a_0 \in \Z^d\setminus \{0\}$.
Let $\eta$ be a Borel probability measure on $X$.
If 
\begin{equation}
\label{eq:fRW}
\abse{\int_X f \dd\mu^{*m} * \eta} \geq t\norm{f}_{0,\alpha},
\end{equation}
for some $m \geq \max(C\log\frac{1}{t}, m_0)$, then there exists $f_1 \in \Ccal^{0,\alpha}(X)\cap \Hcal_0$ such that
\[\abse{\int_X f_1 \dd \eta - \int_X f_1 \dd \mes_X} \geq (2C_\beta)^{-2 m} \, t^2 \, \norm{f_1}_{0,\alpha}.\]
\end{lemm}

\begin{proof}
Without loss of generality, we may assume $\norm{f}_{0,\alpha} = 1$.

We are going the partition $\Gamma\ltimes N$ according to the action on $\Z^d$.
For $a \in \Z^d$, define
\[P_a = \{\, g \in \Gamma \ltimes N \mid  \theta(g)^{-1}\cdot a_0 = a\,\}.\]
For $m\geq 1$ and $a \in \Z^d$ define $\mu^{(m)}_a$ to be renormalized restriction of $\mu^{*m}$ to $P_a$ so that we have 
\[\mu^{*m} = \sum_{a \in \Z^d} \mu^{*m}(P_a) \mu^{(m)}_a.\]
Define also
\[ f^{(m)}_a = U^*(\mu^{(m)}_a) f\]
so that
\begin{equation}
\label{eq:fassum}
U^*(\mu)^m f = \sum_{a \in \Z^d} \mu^{*m}(P_a) f^{(m)}_a.
\end{equation}
From the definition of $P_a$, we know that $f^{(m)}_a \in \Hcal_a$. Hence the sum in \eqref{eq:fassum} is an orthogonal one. In particular,
\begin{equation}
\label{eq:fassum'}
    \normbig{U^*(\mu)^m f}_2^2 = \sum_{a \in \Z^d} \mu^{*m}(P_a)^2 \norm{f^{(m)}_a}_2^2.
\end{equation}
From $f^{(m)}_a \in \Hcal_a$, follows that $\absbig{f^{(m)}_a}^2 \in \Hcal_0$.
The functions $\absbig{f^{(m)}_a}^2$ are going to be our candidates for $f_1$.

The core of the argument is the following applications of the Cauchy-Schwarz inequality.
From \eqref{eq:fRW} and \eqref{eq:fassum}, we get
\[t \leq \abse{\int_X U^*(\mu)^m f \dd\eta} \leq \sum_{a \in \Z^d} \mu^{*m}(P_a) \abse{\int_X f^{(m)}_a \dd \eta}.\]
By the Cauchy-Schwarz inequality for the sum,
\[t^2 \leq \sum_{a \in \Z^d} \mu^{*m}(P_a) \abse{\int_X f^{(m)}_a \dd \eta}^2.\]
By the Cauchy-Schwarz inequality for the integral,
\begin{equation}
\label{eq:t2fma}
t^2 \leq \sum_{a \in \Z^d} \mu^{*m}(P_a) \int_X \absbig{f^{(m)}_a}^2 \dd \eta.
\end{equation}
We want to compare the right hand side of the above equation (where the integration is over the unknown measure $\eta$) to the following analogous expression involving $\norm{\cdot}_2^2$, i.e.\ when the integration is with respect to the Haar measure $\mes_X$:
\[\sum_{a \in \Z^d} \mu^{*m}(P_a) \normbig{f^{(m)}_a}_2^2 = \sum_{a \in \Z^d} \mu^{*m}(P_a) \int_X \absbig{f^{(m)}_a}^2 \dd \mes_X.\]

But first, we need throw away the $a$'s for which the $\normbig{f^{(m)}_a}_{0,\alpha}$ is too large.
From the exponential moment assumption and \eqref{eq:product_Lip},
\[\sum_{a \in \Z^d} \mu^{*m}(P_a) \int_{P_a} \Lip_X(g)^\beta \dd \mu^{(m)}_a(g) =  \int_{\Aff(X)} \Lip_X(g)^\beta \dd \mu^{*m}(g) \leq C_\beta^m.\]
By the Markov inequality, we have for any $\kappa > 0$,
\[ \mu^{*m}*\delta_{a_0}(B_\kappa) \leq e^{-\kappa m}\]
where
\[B_\kappa = \Bigl\{\, a \in \Z^d \mid \int_{P_a} \Lip_X(g)^\beta \dd \mu^{(m)}_a(g) > C_\beta^{m}e^{\kappa m} \,\Bigr\}.\]
For $a \in \Z^d \setminus B_\kappa$, we have, since $\alpha \leq \beta$,
\begin{align}
\normbig{f^{(m)}_a}_{0,\alpha} &\leq \int_{P_a} \norm{U^*(g)f}_{0,\alpha} \dd \mu^{(m)}_a(g)\notag\\
&\leq \int_{P_a} \Lip_X(g)^\alpha \norm{f}_{0,\alpha} \dd \mu^{(m)}_a(g)\notag\\
\label{eq:Bknorm0alpha} &\leq C_\beta^m e^{\kappa m}.
\end{align}

Next, we need to exploit the assumption $\tau_Z(\mu) < 2\sigma_{X,Y}(\mu)$.
Choose $\tau > \tau_Z(\mu)$ and $\sigma < \sigma_{X,Y}(\tau)$ such that $\sigma - \tau/2 = \frac{2\sigma_{X,Y}(\mu) - \tau_Z(\mu)}{4} > 0$ and moreover,
\begin{enumerate}
\item for $m \geq m_1 = m_1(\mu)$, $\norm{U_{X,Y}(\mu)^m} \leq e^{-\sigma m}$;
\item there exists $\kappa = \kappa(\theta_*\mu) \in {(0,\frac{1}{2})}$ such that for  $m \geq m_2 = m_2(\theta_*\mu)$, there exists $A \subset \Z^d$ satisfying $\# A \leq e^{\tau m}$ and $\mu^{*m}*\delta_{a_0}(A) \geq 1 - e^{-\kappa m}$.
\end{enumerate}
Note that $\kappa$ only depends on $\theta_*\mu$ because we are letting $\langle \Supp(\mu)\rangle$ act on $Z$ via $\langle \Supp(\mu)\rangle \xrightarrow{\theta} \Gamma \to \Aut(Z)$.
By replacing $A$ by $A \setminus B_\kappa$, we may assume without loss of generality that $A \subset \Z^d\setminus B_\kappa$.

Using the fact that $\norm{U^*(\mu^{(m)}_a)} \leq 1$ (hence $\normbig{f^{(m)}_a}_2 \leq \norm{f}_2 \leq 1$), the Cauchy-Schwarz inequality, and \eqref{eq:fassum'}, we obtain 
\begin{align*}
\sum_{a \in A} \mu^{*m}(P_a) \normbig{f^{(m)}_a}_2^2 &\leq \sum_{a \in A} \mu^{*m}(P_a) \normbig{f^{(m)}_a}_2\\
&\leq \sqrt{\# A} \sqrt{\sum_{a \in A} \mu^{*m}(P_a)^2 \normbig{f^{(m)}_a}_2^2} \\
&= \sqrt{\# A} \normbig{U^*(\mu)^m f}_2\\
&\leq \sqrt{\# A} \normbig{U_{X,Y}(\mu)^m}\\
&\leq e^{-(\sigma - \tau/2)m}.
\end{align*}

Now remember \eqref{eq:t2fma}. Bounding $\norm{f^{(m)}_a}_\infty \leq \norm{f}_\infty \leq 1$ for $a \in \Z^d \setminus A$, we obtain
\[t^2 \leq e^{-\kappa m} + \sum_{a \in A} \mu^{*m}(P_a) \int_X \absbig{f^{(m)}_a}^2 \dd \eta,\]
which we rewrite as 
\[t^2 \leq e^{-\kappa m} + \sum_{a \in A} \mu^{*m}(P_a) \normbig{f^{(m)}_a}_2^2 + \sum_{a \in A} \mu^{*m}(P_a) \int_X \absbig{f^{(m)}_a}^2 \dd (\eta - \mes_X).\]
Then it follows from the above that 
\[t^2 \leq e^{-\kappa m} +  e^{-(\sigma - \tau/2)m} + \sum_{a \in A} \mu^{*m}(P_a) \int_X \absbig{f^{(m)}_a}^2 \dd (\eta - \mes_X),\]
Now if
\[m \geq 2\max\Bigl\{\frac{1}{\kappa},\frac{1}{\sigma - \tau/2}\Bigr\} \log \frac{2}{t}\]
then
\[\frac{t^2}{2} \leq \sum_{a \in A} \mu^{*m}(P_a) \int_X \absbig{f^{(m)}_a}^2 \dd (\eta - \mes_X).\]
Hence there exists $a \in A$ such that
\[\int_X \absbig{f^{(m)}_a}^2 \dd (\eta - \mes_X)\geq \frac{t^2}{2}.\]
Moreover, since $A \subset \Z^d\setminus B_\kappa$, we have by \eqref{eq:product_Holder} and \eqref{eq:Bknorm0alpha},
\[\normbig{\absbig{f^{(m)}_a}^2}_{0,\alpha} \leq \normbig{f^{(m)}_a}_{0,\alpha}^2 \leq C_\beta^{2m} e^{2 \kappa m} \leq (2C_\beta)^{2m}.\]
Thus, $f_1 = \absbig{f^{(m)}_a}^2$ satisfies the required properties, proving the lemma with $m_0 = \max\{m_1,m_2\}$ and $C = 4\max\Bigl\{\frac{1}{\kappa},\frac{4}{2\sigma_{X,Y}(\mu) - \tau_Z(\mu)}\Bigr\}$.
\end{proof}

\subsection{Proof of the key proposition}
We need one more lemma before we prove Proposition~\ref{pr:XtoY}. 
\begin{lemm}
\label{lm:gobackm}
Assume that $\mu$ has a finite $\beta$-exponential moment. 
For every $0 < \alpha \leq \min\{1,\beta\}$, there exists a constant $C \geq 1$ such that the following holds for any parameter $t \in {(0,1/2)}$ and any $m \in \N$ sufficiently large. If there exists $f_0 \in C^{0,\alpha} \cap \Hcal_0$ satisfying
\begin{equation*}
\abse{\int_X f_0 \dd \mu^{*m} * \eta - \int_X f_0 \dd \mes_X} \geq t \norm{f_0}_{0,\alpha},
\end{equation*}
then there exists $f_1 \in C^{0,\alpha} \cap \Hcal_0$ such that
\begin{equation*}
\abse{\int_X f_1 \dd \eta - \int_X f_1 \dd \mes_X} \geq e^{-Cm} t^C \norm{f_1}_{0,\alpha}.
\end{equation*}
\end{lemm}

\begin{proof}
Without loss of generality, assume $\norm{f_0}_{0,\alpha} = 1$.
By the moment assumption, there is $C_\beta \geq 1$ such that for any $m \in \N$, $\int \Lip_X(g)^\beta \dd \mu^{*m}(g) \leq C_\beta^m$.
Set 
\[E = \{\, g \in \Aff(X) \mid \Lip_X(g)^\beta > 4C_\beta^{m} t^{-1}\,\}\]
so that we have \[\mu^{*m}(E)\leq \frac{t}{4}\] by the Markov inequality . Thus for any $\alpha \in {(0,\beta]}$,  
\[\forall g \in \Aff(X) \setminus E,\quad \norm{U^*(g)f_0}_{0,\alpha} \leq \Lip_X(g)^{\alpha}\norm{f_0}_{0,\alpha} \leq  4C_\beta^{m} t^{-1}.\]
By the assumption on $f_0$, 
\begin{align*}
t &\leq \int_{\Aff(X)} \abse{\int_X U^*(g) f_0 \dd (\eta- \mes_X)} \dd \mu^{*m}(g)\\
&\leq 2 \mu^{*m}(E) + \int_{\Aff(X) \setminus E} \abse{\int_X U^*(g) f_0 \dd (\eta- \mes_X)} \dd \mu^{*m}(g).
\end{align*}
Hence
\[\int_{\Aff(X) \setminus E} \abse{\int_X U^*(g) f_0 \dd (\eta- \mes_X)} \dd \mu^{*m}(g) \geq \frac{t}{2}.\]
Hence there exists $g \in \Aff(X) \setminus E$ such that $f_1 = U^*(g) f_0$ satisfies
\[\abse{\int_X f_1 \dd (\eta- \mes_X)} \geq \frac{t}{2}.\]
Moreover, since $g \notin E$,
\[\norm{f_1}_{0,\alpha} \leq 4C_\beta^{m} t^{-1},\]
showing the required property for $f_1$.
\end{proof}

\begin{proof}[Proof of Proposition~\ref{pr:XtoY}]
Let $t \in {(0,1)}$ be such that there exists $f \in \Ccal^{0,\alpha}(X)$ such that 
\[\abse{\int_X f \dd\mu^{*m}*\eta - \int_X f \dd \mes_X} \geq t \norm{f}_{0,\alpha}.\]
By Lemma~\ref{lm:finda0} there is $a_0 \in \Z^d$ and $f_0 \in \Ccal^{0,\alpha}(X) \cap \Hcal_{a_0}$ such that
\[\abse{\int_X f_0 \dd\mu^{*m}*\eta - \int_X f_0 \dd \mes_X} \geq t^{O(1)} \norm{f_0}_{0,\alpha}.\]

Using either Lemma~\ref{lm:gobackm} in the case where $a=0$ or Lemma~\ref{lm:CSargument} otherwise (note that $a_0 \neq 0$ implies $\int_X f_0 \dd \mes_X= 0$), we obtain some $f_1 \in \Ccal^{0,\alpha}(X) \cap \Hcal_{0}$ such that
\begin{equation}
\abse{\int_X f_1 \dd\eta - \int_X f_1 \dd \mes_X} \geq e^{-O(m)} t^{O(1)} \norm{f_1}_{0,\alpha}.
\end{equation}
 
Letting $\phi \in \Ccal^{0,\alpha}(Y)$ be such that $f_1 = \phi \circ \pi$, we have $\int_X f_1 \dd \eta = \int_Y \phi \dd \pi_* \eta$,  $\int_X f_1 \dd \mes_X = \int_Y \phi \dd \mes_Y$ and $\norm{\phi}_{0,\alpha} \ll \norm{f_1}_{0,\alpha}$. The last implied constant depends only on the choice of Riemannian metrics on $X$ and on $Y$.
Therefore,
\[\abse{\int_Y \phi \dd\pi_* \eta - \int_Y \phi \dd \mes_Y} \gg e^{-O(m)} t^{O(1)} \norm{\phi}_{0,\alpha},\]
finishing the proof of the proposition.
\end{proof}

\section{Proof of the main theorems}
\label{sc:main proof}
We are ready to prove the main theorem of this paper.
\begin{proof}[Proof of Theorem~\ref{thm:main}]
We use the same notation $\mu$, $\beta$, $X$, $N$, $\Lambda$, $Z$, $Y$, $\Gamma$, $\theta$ in Proposition~\ref{pr:XtoY} as in Theorem~\ref{thm:main}. 

Assume that the $\mu$-induced walk on $Y$ satisfies $(C_Y,\lambda,\alpha)$-quantitative equidistribution for parameters $C_Y > 0$, $\lambda > 0$ and $0 < \alpha \leq \min(1,\beta)$. 
Let $\lambda' \in {(0,\lambda)}$. 
We want to show that the $\mu$-induced walk on $X$ satisfies $(C_X,\lambda',\alpha)$\dash{}quantitative equidistribution for a large constant $C_X$.
Assume that for some $t \in {(0,\frac{1}{2})}$, $m \geq C_X \log\frac{1}{t}$ it holds that $\Wcal_\alpha(\mu^{*m} * \delta_x,\mes_X) > t$. 

Denote by $\pi \colon X \to Y$ the natural projection.
Now we can apply Proposition~\ref{pr:XtoY}, whose constant we denote by $C_\pi$, on $\eta = \mu^{*(m - m')}*\delta_x$ with $m'$ random walk steps.
Choose $m'$ to be such that 
\[C_\pi \log\frac{1}{t} < m' <2 C_\pi \log\frac{1}{t}.\]
By the proposition,
\[
\Wcal_\alpha(\mu^{*(m-m')}*\delta_{\pi(x)},\mes_Y) =
\Wcal_\alpha(\pi_*\eta,\mes_Y) \ge e^{-C_\pi m'} > t^{2 C_\pi^2},
\]
If $C_X$ is large enough then we can guarantee that $m - m' \geq C_Y \log(t^{-2 C_\pi^2})$, so that the premise of the $(C_Y,\lambda,\alpha)$\dash{}quantitative equidistribution of the random walk induced on $Y$ applies. 

For simplicity, we will assume for the remainder of the proof that $\mu$ is supported on $\Aut(X)$, and leave the case that it is supported on $\Aff(X)$ to the reader. The two proofs are almost identical.

The quantitative equidistribution on $Y$ tells us that there exists $y'\in Y$ with  $d(\pi(x),y') < e^{-\lambda (m-m')}$ such that the projection of $\Gamma y'$ to the maximal torus factor $T_Y$ of $Y$ is contained in a proper closed $\Gamma$-invariant subgroup $L$ of $T_Y$ of height $\leq t^{-2 C_Y C_\pi^2}$. Note that if $C_X > \frac{2 C_\pi \lambda}{\lambda-\lambda'}$, then $e^{-\lambda (m-m')} < e^{-\lambda' m}$. By choosing $x'\in X$ to be the point closest to $x$ in $\pi^{-1}(y')$, we get $d(x,x')= d(\pi(x),y')<e^{-\lambda' m}$.

Let $T_X$ denote the maximal torus factor of $X$. Then $T_Y$ is a factor of $T_X$. Moreover the following the diagram of $\Gamma$-equivariant maps commutes.
\[
  \begin{tikzcd}
    X \arrow{r}{\pi} \arrow{d} & Y \arrow{d} \\
    T_X \arrow{r}{\pi'} & T_Y
  \end{tikzcd}
\]
Thus, the projection of $\Gamma x'$ to $T_X$ is contained in $\pi'^{-1}(L)$, which is a 
proper closed $\Gamma$-invariant subgroup of $T_X$ of height $\leq O(t^{-2 C_Y C_\pi^2})$ by the following observation. 
\begin{lemm}
Let $T'$ be a factor torus of a torus $T$ and let $\pi' \colon T \to T'$ be the factor map. There exists $C' > 1$ such that for if $L$ is a proper closed subgroup of $T'$ of height $\le h$, then $\pi'^{-1}(L)$ is a proper closed subgroup of $T$ of height $\leq C' h$.
\end{lemm}
\begin{proof}
A generating set of the dual of $L$ can be mapped by the dual of $\pi'$ to a generating set of the dual of $\pi'^{-1}(L)$. The dual of $\pi'$ changes the norm of the vectors by at most a finite factor $C'$, the operator norm of this linear transformation.
\end{proof}

By taking a new $C_X$ that is large enough, this give us $(C_X,\lambda',\alpha)$\dash{}quantitative equidistribution of the random walk on $X$.
\end{proof}

Theorem~\ref{thm:tower} follows immediately.

\begin{proof}[Proof of Theorem~\ref{thm:tower}]
Remark that for each $k = 1,\dotsc, l - 1$, because $X \to X_k$ is a smooth map between compact Riemannian manifolds, the condition that $\mu$ has a finite exponential moment implies that the image measure in $\Aff(X_k)$ also has finite exponential moment.
It suffices then to use Theorem~\ref{thm:torus} for the random walk on the torus $X_{l-1}$ and apply Theorem~\ref{thm:main} repeatedly $l - 1$ times. 
\end{proof}

Corollary~\ref{cor:qualitative main} follows from the following lemma.

\begin{lemm}
Let $X$ be a nilmanifold and $\mu$ a probability measure on $\Aff(X)$.
Let $T$ denote the maximal torus factor of $X$.
Let $H$ denote the subgroup generated by $\Supp(\mu)$.
If the $\mu$-induced random walk on $X$ satisfies a $(C,\lambda,\alpha)$\dash{}quantitative equidistribution for some $C > 0$, $\lambda > 0$ and $\alpha \in (0,1]$ then for any $x \in X$
\begin{enumerate}
    \item either $\mu^{*m}*\delta_x$ converges to $\mes_X$ in the weak-$*$ topology, 
    \item or the projection of $Hx$ to $T$ is contained in a proper closed $H$-invariant subset.
\end{enumerate}
\end{lemm}
\begin{proof}
Let $\pi \colon X \to T$ be the projection from $X$ to its maximal torus factor.
Assume that  $\mu^{*m}*\delta_x$ does not converge to $\mes_X$ in the weak-$*$ topology.
The space of $\alpha$-Hölder functions $C^{0,\alpha}(X)$ is dense in the space of continuous functions.
It follows that there is $t > 0$ such that 
\[ \Wcal_\alpha(\mu^{*m}*\delta_x, \mes_X) > t\]
for an unbounded sequence of $m$.

From the quantitative equidistribution, we get
\begin{enumerate}
\item a sequence $(x_k)$ of points in $X$,
\item a sequence $(H_k)$ of subgroups of $\Aff(X)$,
\item a sequence $(L_k)$ of proper closed subgroup of $T$ of height $\leq t^{-C}$ and invariant under $\theta(H_k) = \theta(H)$
\end{enumerate}
such that 
\begin{enumerate}
\item $\lim_{k \to +\infty} x_k  = x$,
\item $\lim_{k \to +\infty} \sup_{g \in \Supp(\mu)} d(g, H_k) = 0$, and
\item $\pi(H_k x_k) - \pi(x_k) \subset L_k$ for all $k$.
\end{enumerate}
In $T$, there are only finitely many closed subgroup of height $\leq t^{-C}$.
Therefore, after extracting a subsequence, we may assume that $L_k =: L$ are all equal.
Letting $k$ go to $+\infty$, we find
\[\forall g \in \Supp(\mu),\quad  \pi(gx) - \pi(x) \in L.\]
This is enough to conclude that $\pi(x) + L \subset T$ is $H$-invariant and $\pi(Hx) \subset \pi(x) + L$.
\end{proof}


\appendix
\section{A large deviation estimate.}
\label{sc:large dev}
Let $\mu$ be Borel probability measure on $\GL_d(\R)$, $d \geq 2$. 
Consider the random walk in the linear group defined by $\mu$.
Recall that $\mu$ is said to have a finite exponential moment if there is $\beta > 0$ such that
\begin{equation}
\label{eq:expmomentGLd}
\int_{\GL_d(\R)} \max\bigl\{\norm{g}, \norm{g^{-1}}\bigr\}^\beta \dd \mu(g) < +\infty.    
\end{equation}

Recall also that the top Lyapunov exponent of $\mu$ is defined by
\[\lambda_1(\mu) = \lim_{m\to +\infty}\frac{1}{m}\int_{\GL_d(\R)} \log \norm{g} \dd \mu^{*m}(g)\]
\begin{thm}
\label{thm:LD}
Let $\mu$ be a Borel probability measure on $\GL_d(\R)$. Assume $\mu$ has a finite exponential moment. For any $\omega > 0$ there is $\kappa > 0$ such that for all $m$ large enough.
\[\mu^{*m}\{\, g\in \GL_d(\R) \mid \log \norm{g} > m (\lambda_1(\mu) + \omega) \,\} \leq e^{-\kappa m}.\]
\end{thm}
\begin{proof}
Let $g_1, g_2, \dotsc $ be independent random variables distributed according to $\mu$.

Given $\omega > 0$, let $l \geq 1$ be such that
\[\Ebb\bigl[ \log\norm{g_1 \dotsm g_l}\bigr] < l(\lambda_1(\mu) + {\omega}/{3}).\]
Observe that $\bigl(\log \norm{g_{kl-l+1} \dotsm g_{kl}}\bigr)_{k \geq 1}$ is a sequence of i.i.d. real-valued random variables having a finite exponential moment.
Thus, by Cr\'amer's theorem, there exists $\tau > 0$ such that for $k$ large enough,
\[\Pbb\bigl[\log \norm{g_1\dotsm g_l} + \dotsb + \log\norm{g_{kl-l+1}\dotsm g_{kl}} > kl(\lambda_1(\mu) + {2\omega}/{3})\bigr] \leq e^{-\tau k}.\]
The norm is submultiplicative, hence for $k$ large enough,
\[\Pbb\bigl[\log \norm{g_1\dotsm g_{kl}}> kl(\lambda_1(\mu) + {2\omega}/{3})\bigr] \leq e^{-\tau k}.\]

For any $m$, write $m = kl +j$ with $0 \leq j < k$. 
Using submultiplicativity again, we see that if $\log \norm{g_1 \dotsm g_m} > m (\lambda_1(\mu) + \omega)$ then either $\log \norm{g_1\dotsm g_{kl}} > kl(\lambda_1(\mu) + {2\omega}/{3})$ or there is $1\leq i \leq j$ such that $\log\norm{g_{kl+i}} > \frac{\omega}{3l}m$.
Thus,
\[\Pbb\bigl[\log \norm{g_1\dotsm g_{m}}> m(\lambda_1(\mu) + \omega)\bigr]  \leq e^{-\tau k} + l \Pbb\bigl[\log \norm{g_1} > \frac{\omega}{3l}m\bigr].\]
Finally, since $\mu$ has a finite exponential moment, there is some $\beta > 0$ such that $\Ebb\bigl[\norm{g_1}^\beta\bigr]$ is finite.
Hence by Markov's inequality,
\[\Pbb\bigl[\log \norm{g_1} > \frac{\omega}{3l}m\bigr] = \Pbb\bigl[\norm{g_1}^\beta > e^{\frac{\beta\omega}{3l}m}\bigr] \leq e^{-\frac{\beta \omega}{3l}m} \Ebb\bigl[\norm{g_1}^\beta\bigr].\]
Put together, we find 
\[\Pbb\bigl[\log \norm{g_1\dotsm g_{m}}> m(\lambda_1(\mu) + \omega)\bigr]  \leq e^{-\kappa m}\]
for $\kappa = \frac{1}{2}\min\{\frac{\tau}{l}, \frac{\beta \omega}{3l}\}$ and $m$ large enough.
\end{proof}

\section{The case of a torus}
\label{sc:Zconnected}
Here we explain how to remove the Zariski connectedness assumption in the main theorem of \cite{HS}.
Namely, the goal is the following.
\begin{thm}
\label{thm:aut,base}
Let $X = \R^d/\Z^d$.
Let $\mu$ be a probability measure on $\Aut(X) = \GL_d(\Z)$ having a finite exponential moment.
Let $\Gamma$ denote the subgroup generated by the support of $\mu$.
Assume that the action of $\Gamma$ on $\R^d$ is strongly irreducible. 
Then given any $\lambda \in (0,\lambda_{1,\R^d}(\mu))$, there exists a constant $C = C(\mu,\lambda) \geq 1$ such that the following holds. 
If $x \in X$ satisfies
\[\abs{\widehat{\mu^{*n} *\delta_x}(a)} > t \quad \text{and} \quad n \geq C \log\frac{\norm{a}}{t}\]
for some $a \in \Z^d\setminus\{0\}$ and $t \in {(0,\frac{1}{2})}$, then there exists a rational point $x' \in X$ of denominator at most $(\frac{\norm{a}}{t})^{C}$ such that $d(x,x') \leq e^{-\lambda m}$.
\end{thm}

The corresponding statement for affine random walks is the following. 
Recall that $\theta \colon \Aff(X) \to \Aut(X)$ denote the linear part.
\begin{thm}
\label{thm:aff,base}
Let $X = \R^d/\Z^d$.
Let $\mu$ be a finitely supported probability measure on $\Aff(X) = \GL_d(\Z) \ltimes \R^d$.
Let $\Gamma$ denote the subgroup generated by the support of $\theta_*\mu$.
Assume that the action of $\Gamma$ on $\R^d$ is strongly irreducible. 
Then given any $\lambda \in (0,\lambda_{1,\R^d}(\mu))$, there exists a constant $C = C(\theta_*\mu,\lambda) \geq 1$ such that the following holds. 
If $x \in X$ satisfies
\[\abs{\widehat{\mu^{*n} *\delta_x}(a)} > t \quad \text{and} \quad n \geq C \log\frac{\norm{a}}{t}\]
for some $a \in \Z^d\setminus\{0\}$ and $t \in {(0,\frac{1}{2})}$, then there exists a point $x' \in X$ and a finite set $F \subset \Aff(X)$ such that $d(x,x') \leq e^{-\lambda m}$, $d_{\mathrm{H}}(\Supp(\mu),F) \leq e^{-\lambda m}$ and moreover, denoting by $H$ the subgroup generated by $F$, the orbit $Hx'$ is finite of cardinality at most $(\frac{\norm{a}}{t})^{C}$.
\end{thm}

In view of \cite[Lemma 4.5]{Boyer} or alternatively Lemma~\ref{lm:finda0}, Theorem~\ref{thm:torus} follows.

The key point is a Fourier decay estimate for $(\theta_*\mu)^{*n}$, stated as Theorem~\ref{thm:decaymun} below, replacing~\cite[Theorem 3.20]{HS}.
To establish this Fourier decay property, we first need a Fourier decay estimate for multiplicative convolutions of measures having nice non-concentration properties, Theorem~\ref{thm:fourier}.
Then in subsection~\ref{ss:returnTime}, using return times and the special case of Zariski-connected groups, we obtain a decomposition of $(\theta_*\mu)^{*n}$ as a sum of multiplicative convolutions of measures having the required non-concentration properties.
Once Theorem~\ref{thm:decaymun} established, the rest of the proof of Theorem~\ref{thm:aut,base} is identical to the corresponding part in~\cite{HS} and that of Theorem~\ref{thm:aff,base} to the corresponding part in~\cite{HLL}.

\subsection{Multiplicative convolutions in simple algebras}
First, we need a slight improvement of \cite[Theorem 2.1]{HS} by allowing the measures we convolve to be different.

Let $E$ be a normed simple algebra over $\R$ of finite dimension.
For $x \in E$, denote by $\det(x)$ the determinant of the linear endomorphism $E \to E$, $y \mapsto xy$.
For $\rho > 0$, write
\[S(\rho) = \{\, x \in E \mid \abs{\det(x)} \leq \rho \,\}.\]

\begin{defn}
Let $\epsilon > 0$, $\kappa > 0$, $\tau > 0$ be parameters.
We say a Borel measure $\eta$ on $E$ satisfies $\NC_0(\epsilon,\kappa,\tau)$ at scale $\delta > 0$ if
\begin{enumerate}
\item $\eta \bigl(E \setminus B(0,\delta^{-\epsilon})\bigr) \leq \delta^{\tau}$;
\item for every $x \in E$, $\eta(x + S(\delta^{\epsilon})) \leq \delta^{\tau}$;
\item for every $\rho \geq \delta$ and every proper affine subspace $W \subset E$, $\eta(W^{(\rho)}) \leq \delta^{-\epsilon} \rho^\kappa$.
\end{enumerate}
\end{defn}

Throughout this appendix, each occurrence of $((t))$ with $t > 0$ denotes an unspecified Borel measure of total mass at most $t$.
\begin{defn}
We say a Borel probability measure $\eta$ on $E$ satisfies $\NC(\epsilon,\kappa,\tau)$ at scale $\delta$ if it can be written as $\eta = \eta_0 + ((\delta^\tau))$ with $\eta_0$ satisfying $\NC_0(\epsilon,\kappa,\tau)$ at scale $\delta$.
\end{defn}

Here, $\NC$ stands for non-concentration. 

Let $E^*$ denote the linear dual of $E$ over $\R$. Recall that the Fourier transform of a finite Borel measure $\nu$ on $E$ is defined as
\[\forall \xi \in E^*,\quad \hat{\nu}(\xi) = \int_E e(\xi(x)) \dd \nu(x)\]
where $e(t) = e^{2\pi i t}$, for $t \in \R$.

\begin{thm}[Fourier decay of multiplicative convolutions in simple algebra]
\label{thm:fourier}
Let $E$ be a normed simple algebra over $\R$ of finite dimension. Given $\kappa > 0$, there exists $s = s(E,\kappa) \in \N$ and $\epsilon = \epsilon(E,\kappa) > 0$ such that for any parameter $\tau \in {(0, \epsilon \kappa)}$ the following holds for any scale $\delta > 0$ sufficiently small.

If $\eta_1, \dotsc, \eta_s$ are Borel probability measures on $E$ satisfying $\NC(\epsilon,\kappa,\tau)$ at scale $\delta$, then for all $\xi \in E^*$ with $\delta^{-1+\epsilon} \leq \norm{\xi} \leq \delta^{-1-\epsilon}$,
\[
\abs{(\eta_1 * \dotsm * \eta_s)^\wedge(\xi)} \leq \delta^{\epsilon \tau}.
\]
\end{thm}
The special case where $\eta_1 = \dots = \eta_s$ are the same measure is precisely \cite[Theorem 2.1]{HS}.
We will deduce the general case from the special case using a trick from an article of Bourgain and Dyatlov~\cite{BD}.


For measures $\eta$ and $\eta'$ on $E$, we write $\eta \pp \eta'$ for the additive convolution between $\eta$ and $\eta'$. 
Similarly, $\eta \mm \eta'$ denotes the image measure of $\eta \otimes \eta'$ under the map $(x,y) \mapsto x - y$.
Finally, for integer $k \geq 1$, we write 
\[\eta^{\pp k} = \underbrace{\eta \pp \dotsb \pp \eta}_{k \text{ times}}.\]


The following two observations on the $\NC$ property are immediate.
\begin{lemm}
\label{lm:ppffNC}
Let $\epsilon, \kappa, \tau, \sigma > 0$ be parameters and let $\delta > 0$.
\begin{enumerate}
\item If $\eta$ is a Borel probability measures on $E$ satisfying $\NC(\epsilon,\kappa,\tau)$ at scale $\delta$, then $\eta \mm \eta$ 
satisfies $\NC(O(\epsilon),\kappa,\tau/2)$ at scale $\delta$.
\item Convex combinations of probability measures satisfying $\NC(\epsilon,\kappa,\tau)$ also satisfy $\NC(\epsilon,\kappa,\tau)$ at scale $\delta$.
\end{enumerate}

\end{lemm}

\begin{lemm}
\label{lm:eta'NC}
Let $\epsilon, \kappa, \tau, \sigma > 0$ be parameters and let $\delta > 0$.
Let $\eta$ and $\eta'$ be Borel probability measures on $E$ such that $\eta = \delta^\sigma \eta' + ((1))$. 
If $\eta$ satisfies $\NC(\epsilon,\kappa,\tau)$ at scale $\delta$ then $\eta'$ satisfies $\NC(\epsilon + \sigma,\kappa,\tau - \sigma)$ at scale $\delta$.
\end{lemm}

Finally, we will need to compare Fourier transform of multiplicative convolutions with that of multiplicative convolutions of additive convolutions.
\begin{lemm}
\label{lm:nu123}
Let $\nu$, $\nu'$, $\nu''$ be Borel probability measures on $E$, then for any integer $k \geq 1$, the Fourier transform of $\nu * (\nu'^{\pp k} \mm \nu'^{\pp k}) * \nu''$ takes non-negative real values and moreover,
\[\forall \xi \in E^*,\quad \abs{(\nu * \nu' * \nu'')^{\wedge}(\xi)}^{2k} \leq \bigl(\nu * (\nu'^{\pp k} \mm \nu'^{\pp k}) * \nu''\bigr)^{\wedge}(\xi).\]
\end{lemm}

\begin{proof}
By definition,
\[(\nu * \nu' * \nu'')^{\wedge}(\xi) = \iiint e(\xi(xyz)) \dd \nu(x) \dd \nu'(y) \dd \nu''(z).\]
By Hölder's inequality applied to the function $(x,z) \mapsto \int e(\xi(xyz)) \dd \nu'(y)$,
\begin{align*}
&\abs{(\nu * \nu' * \nu'')^{\wedge}(\xi)}^{2k} \\
\leq& \iint \abse{\int e(\xi(xyz)) \dd \nu'(y)}^{2k}  \dd \nu(x)  \dd \nu''(z)\\
=& \iiint e\bigl(\xi\bigl(x(y_1 + \dotsb + y_k - y_{k+1} - \dotsb - y_{2k})z\bigr)\bigr)\\ 
&\hspace{12em} \dd \nu'^{\otimes 2k}(y_1,\dotsc,y_{2k})  \dd \nu(x)  \dd \nu''(z)\\
=& \bigl(\nu * (\nu'^{\pp k} \mm \nu'^{\pp k}) * \nu''\bigr)^{\wedge}(\xi).\qedhere
\end{align*}
\end{proof}

\begin{proof}[Proof of Theorem~\ref{thm:fourier}]
For $\lambda = (\lambda_1,\dotsc,\lambda_s) \in \C^s$, define
\[ \eta_\lambda = \lambda_1 \eta_1 \mm \eta_1 + \dotsb + \lambda_s \eta_s \mm \eta_s.\]
Consider the function $F \colon \C^s \to \C$ defined by
\[F(\lambda) = \widehat{\eta_\lambda^{*s}}(\xi) = (\eta_\lambda * \dotsm * \eta_\lambda)^\wedge(\xi).\]
For all $\lambda = (\lambda_1,\dotsc,\lambda_s) \in \R^s$ with $\lambda_1 + \dotsb + \lambda_s = 1$, by Lemma~\ref{lm:ppffNC}, $\eta_\lambda$ satisfy $\NC(\epsilon,\kappa,\tau/2)$ at scale $\delta$. 
Hence by \cite[Theorem 2.1]{HS}, we can bound
\[\abs{F(\lambda)} \leq \delta^{\epsilon_0 \tau}\]
for some $\epsilon_0 = \epsilon_0(E,\kappa)$.

Observe that $F(\lambda)$ is a homogeneous polynomial function of degree $s$.
Then above implies 
\[\abs{\partial_1 \dotsm \partial_s F(0,\dotsc,0)} \ll \delta^{\epsilon_0 \tau}.\]
The left-hand side is the coefficient of the monomial term $\lambda_1 \dotsm \lambda_s$, which is
\[\partial_1 \dotsm \partial_s F(0,\dotsc,0) = \sum_{\sigma \in \mathfrak{S}_s}  \bigl((\eta_{\sigma(1)} \mm \eta_{\sigma(1)}) * \dotsm * (\eta_{\sigma(s)} \mm \eta_{\sigma(s)}) \bigr)^\wedge(\xi).\]
By Lemma~\ref{lm:nu123}, each term of the right-hand side is non-negative real. 
It follows that
\[\absbig{\bigl((\eta_{1} \mm \eta_{1}) * \dotsm * (\eta_{s} \mm \eta_{s}) \bigr)^\wedge(\xi)} \ll \delta^{\epsilon_0 \tau}.\]
In view of Lemma~\ref{lm:nu123}, this concludes the proof of the theorem.
\end{proof}

\subsection{Fourier decay for linear random walks.}
\label{ss:returnTime}
From now on let $\mu$ be a probability measure on $\Aut(\Tbb^d) = \GL_d(\Z)$.
Let $\lambda_1$ denote the top Lyapunov exponent of $\mu$ and 
let $\Gamma$ denote the subgroup generated by $\Supp(\mu)$.
We assume
\begin{enumerate}
\item The measure $\mu$ has a finite exponential moment.
\item The action of $\Gamma$ on $\R^d$ is strongly irreducible.
\end{enumerate}
Let $G$ denote the Zariski closure of $\Gamma$ in $\GL_d$ and $G^\circ$ the identity component of $G$; then $\Gamma_0 = \Gamma\cap G^\circ$ is a finite index subgroup of $\Gamma$.
Let $E$ denote the subalgebra generated by $G^\circ(\R)$. If $\gamma_1,\dots,\gamma_J$ are a complete set of representatives for the cosets in $\Gamma/\Gamma_0$, then for any $\gamma \in \Gamma$ we have that $\gamma E = \gamma_j E$ for some $1\leq j \leq J$.

\begin{thm}[Fourier decay for random walks in $\GL_d(\Z)$]
\label{thm:decaymun}
Let $\Gamma$, $\mu$ and $\gamma_1,\dots, \gamma_J$ be as above. Then there exists $\alpha_0 = \alpha_0(\mu) > 0$ such that for every $\alpha \in (0,\alpha_0)$, there exists $c = c(\mu, \alpha) > 0$ such that for all $n$ sufficiently large, all $1 \leq j \leq J$ and $\xi \in E^*$ with \[e^{\alpha n} \leq e^{\lambda_1 n}\norm{\xi} \leq e^{\alpha_0 n}\] the following estimate on Fourier coefficients of $\mu^{*n}$ holds:
\[ \Bigl\lvert \int_{\gamma_j E} e\bigl(\xi(\gamma_j^{-1}g)\bigr) \dd\mu^{*n}(g) \Bigr\rvert  \leq e^{-c_0n}.\]
\end{thm}


Let $(g_n)_{n \geq 1}$ be a sequence of independent random variables distributed according to $\mu$.
Consider the return times to $G^\circ$, 
\[\tau(1) = \inf \{\, n \geq 1 \mid g_n \dotsm g_1 \in G^\circ \,\}\]
and recursively for $m \geq 2$,
\[\tau(m) = \inf \{\, n > \tau(m) \mid g_n \dotsm g_1 \in G^\circ \,\}.\]
They are the return times of a Markov chain on the finite space $G/G^\circ$.
Thus for every $m \geq 1$, $\tau(m)$  is almost surely finite.

Let $\mu^\circ$ denote the law of $g_{\tau(1)} \dotsm g_1$, which is a probability measure on $G^\circ$.
It has the following properties.
\begin{lemm}[{\cite[Lemma 4.40]{Aoun2011}}]
If $\mu$ has a finite exponential moment, then so does~$\mu^\circ$.
\end{lemm}

Denote $T = \Ebb[\tau(1)]$.
Let $\lambda_1 = \lambda_1(\mu)$ denote the top Lyapunov exponent of $\mu$.
\begin{lemm}[{\cite[Lemma 4.42]{Aoun2011}}]
\label{lm:mucirclambda}
The top Lyapunov exponent of $\mu^\circ$ is
\[\lambda_1(\mu^\circ) = T \lambda_1.\]
\end{lemm}

\begin{lemm}[{\cite[Lemma 4.42]{Aoun2011}}]
\label{lm:LDreturn}
Given $\omega > 0$, there is $c = c(\mu,\omega) > 0$ such that for all $m$ sufficiently large, 
\[\Pbb\bigl[ \abs{\tau(m) - T m} \geq \omega m\bigr] \leq e^{-cm}.\]
\end{lemm}

Note that the support of $\mu^\circ$ generates $\Gamma \cap G^\circ$, whose Zariski closure is $G^\circ$.
For $m \geq 1$, in view of Lemma~\ref{lm:mucirclambda}, it is appropriate to rescale $(\mu^\circ)^{*m}$ by a factor of $e^{-T\lambda_1 m}$. Put
\[\tilde\mu^\circ_m =(e^{-T\lambda_1 m})_*(\mu^\circ)^{*m}.\]
Under the assumptions recalled at the beginning of the paragraph, $G^\circ$ acts irreducibly on $\R^d$ and is Zariski-connected.
Thus, we can apply the results in \cite[Section 3]{HS} to the random walk defined by $\mu^\circ$.
As explained in \cite[Proof of Theorem 3.20]{HS}, Proposition~3.1 and Proposition~3.2 of \cite{HS} imply the following.
\begin{lemm}
\label{lm:muppD}
Write $D = \dim E$. There exists $\kappa = \kappa(\mu^\circ) > 0$ such that given any $\alpha > 0$ and $\epsilon > 0$ there exists $\tau > 0$ such that 
the additive convolution $(\tilde\mu^\circ_m)^{\pp D} \mm (\tilde\mu^\circ_m)^{\pp D}$ satisfies $\NC(\epsilon,\kappa,\tau)$ in $E$ at all scales $\delta \in {[e^{-m},e^{-\alpha m}]}$ for all $m \geq 1$ sufficiently large.
\end{lemm}

For $m \geq 1$ and $l \geq 1$, we define $\nu_l$ to be the law of the variable
\[g_{\tau(m)} \dotsm g_1 \quad \text{conditional to the event} \quad \tau(m) = l.\]
By this definition,
\begin{equation}
\label{eq:plnul}
(\mu^\circ)^{*m} = \sum_{l \in \N} p_l \nu_l.
\end{equation}
where $p_l = \Pbb[\tau(m) = l]$.
Here, we are hiding the dependency of $\nu_l$ and $p_l$ on $m$ in order to make the notations less cumbersome. 

Let $n,s$ and $l_1,\dotsc, l_s$ be integers.
Consider the events $\tau(jm) = l_1 + \dotsb + l_j$, $j = 1,\dotsc,s$.
By the Markov property, we have
\[\Pbb\bigl[\forall j  = 1,\dotsc,s,\, \tau(jm) = l_1 + \dotsb + l_j  \bigr] = p_{l_1} \dotsm p_{l_s}\]
Now assume that $l_1 + \dotsb + l_s + k =  n$ with $k \geq 0$ and condition the variable $g_n \dotsm g_1$ according to the events above.
We obtain a decomposition
\begin{equation}
\label{eq:RWcondRT}
\mu^{*n} = \sum_{l_1 + \dotsb + l_s + k = n} p_{l_1} \dotsm p_{l_s} \mu^{*k} * \nu_{l_s} * \dotsm * \nu_{l_1} + ((\Pbb[\tau(sm) > n])).
\end{equation}

With these preparations, the proof of Theorem~\ref{thm:decaymun} is not difficult.
\begin{proof}[Proof of Theorem~\ref{thm:decaymun}]
Let $\alpha > 0$ be given. 
In this proof, each occurence of $c$ denotes a small positive constant depending on $\mu$ and $\alpha$ but independent of $n$.

Let $\kappa = \kappa(\mu^\circ) > 0$ be the constant given by Lemma~\ref{lm:muppD}.
Let $s = s(E,\kappa) \geq 1$ and $\epsilon = \epsilon(E,\kappa) > 0$ be the constants given by Theorem~\ref{thm:fourier}.
By Lemma~\ref{lm:muppD}, there exists $\tau > 0$ such that  $(\tilde\mu^\circ_m)^{\pp D} \mm (\tilde\mu^\circ_m)^{\pp D}$ satisfies $\NC(\epsilon/2,\kappa,2\tau)$ in $E$ at all scales $\delta \in {[e^{-m},e^{-\alpha m/2}]}$, provided that $m \geq 1$ is large enough.
Without loss of generality, we may assume $\tau < \kappa \epsilon$, $\tau < \epsilon /2$.

Let $\omega = \omega(\mu,\alpha)$ to be a constant whose value is to be determined later.
For $n \geq 1$, and choose $m = \left\lfloor \frac{n}{(T + \omega)s} \right\rfloor$. 
Everything below is true for $n$ sufficiently large (larger than some $n_0$ depending on $\mu$ and $\alpha$).
By Lemma~\ref{lm:LDreturn}, we have
\[\Pbb[\tau(sm) > n] \leq e^{-cn}\]
and 
\[\Pbb[\tau(sm) < n - 3 \omega n ] \leq e^{-cn}.\]
Put 
\[\Lcal = \{\, l \in \N \mid p_l \geq e^{-\frac{\alpha\tau}{4D}m} \,\}.\]
We can bound 
\[ \sum_{(l_1,\dotsc,l_s) \not\in \Lcal^s} p_{l_1} \dotsm p_{l_s} \leq sn e^{-\frac{\alpha\tau}{4D}m} \leq e^{-cn}.\]
Thus, \eqref{eq:RWcondRT} becomes
\[ \mu^{*n} = \sum_{\substack{l_1,\dotsc,l_s \in \Lcal, k \leq 3 \omega n \\ l_1 + \dotsb + l_s + k = n}}  p_{l_1} \dotsm p_{l_s} \mu^{*k} * \nu_{l_s} * \dotsm * \nu_{l_1}  + ((e^{-cn})).\]
Let $\gamma$ be one of $\gamma_1,\dots, \gamma_J$. To finish the proof of the theorem, it is suffices establish an upper bound for the quantity
\[I_{l_1,\dotsc,l_s,k}(\xi) :=  \int_{\gamma E} e\bigl(\xi (\gamma^{-1}g)\bigr) \dd \bigl( \mu^{*k} * \nu_{l_s} * \dotsm * \nu_{l_1} \bigr) (g)\]
uniformly for all $l_1,\dotsc,l_s \in \Lcal$, $k \leq 3 \omega n$ with $l_1 + \dotsb + l_s + k = n$.

Indeed, developping $(\tilde\mu^\circ_m)^{\pp D} \mm (\tilde\mu^\circ_m)^{\pp D}$ using \eqref{eq:plnul}, we see
\[(\tilde\mu^\circ_m)^{\pp D} \mm (\tilde\mu^\circ_m)^{\pp D} = p_l^{2D} (e^{-T\lambda_1 m})_* \bigl(\nu_l^{\pp D} \mm \nu_l^{\pp D}\bigr) + ((1)).\]
Since $p_l^{2D} \geq e^{-\alpha \tau m/2} \geq \delta^\tau$ for $l \in \Lcal$, it follows from Lemma~\ref{lm:eta'NC} that $(e^{-T\lambda_1 m})_* \bigl(\nu_l^{\pp D} \mm \nu_l^{\pp D}\bigr)$ satisfies $\NC(\epsilon,\kappa,\tau)$ at all scales $\delta \in {[e^{-m},e^{-\alpha m/2}]}$, provided that $m \geq 1$ is large enough.

Theorem~\ref{thm:fourier} tells us that for $(l_1,\dotsc,l_s) \in \Lcal^s$, for all $\xi \in E^*$ with $e^{\alpha m/2} \leq  e^{T\lambda_1 sm} \norm{\xi} \leq e^m$,
\[\Bigl\lvert \Bigl( \bigl(\nu_{l_s}^{\pp D} \mm \nu_{l_s}^{\pp D}\bigr) * \dotsm * \bigl(\nu_{l_1}^{\pp D} \mm \nu_{l_1}^{\pp D}\bigr) \Bigr)^\wedge(\xi) \Bigr\rvert \leq e^{-\alpha \epsilon \tau m/2}.\]
Using Lemma~\ref{lm:nu123} repeatedly $s$ times, we obtain, for all $\xi \in E^*$ in the same range,
\[\bigl\lvert \bigl(  \nu_{l_s} * \dotsm * \nu_{l_1} \bigr)^\wedge(\xi) \bigr\rvert \leq e^{- \frac{\alpha \epsilon \tau}{2(2D)^s} m} \leq e^{-cn},\]

Let $E$ acts on $E^*$ on the right by
\[\forall x, y \in E,\; \forall \xi \in E^*,\quad (\xi \cdot x)(y) = \xi(xy).\]
For every $\gamma \in \Gamma$ and every $\xi \in E^*$, we have 
\[I_{l_1,\dotsc,l_s,k}(\xi) = \int_{\gamma E} (\nu_{l_s} * \dotsm * \nu_{l_1})^\wedge(\xi \cdot \gamma^{-1}g) \dd \mu^{*k}(g).\]
Note that for any $g \in \gamma E \cap \Gamma$, 
\begin{equation}
\label{eq:xigammag}
\norm{\xi} \norm{g^{-1}}^{-1} \ll_\gamma \norm{\xi \cdot \gamma^{-1} g} \ll_\gamma \norm{\xi} \norm{g}.
\end{equation}

Using the assumption that $\mu$ has a finite exponential moment and Markov's inequality, we can find a constant $C = C(\mu) \geq 1$ such that for any $k \geq 1$, the $\mu^{*k}$-measure of the set of $g \in \Gamma$ such that 
\begin{equation}
\label{eq:goodmuk}
\norm{g} \leq e^{Ck} \quad \text{and}\quad \norm{g^{-1}} \leq  e^{Ck}
\end{equation}
is at least $1 - e^{-k}$.

Set $\alpha_0 = \frac{1}{4Ts}$ and let $\xi \in E^*$ be such that $e^{\alpha n} \leq e^{\lambda_1 n}\norm{\xi} \leq e^{\alpha_0 n}$. Using $(1 - 2\omega) n \leq Tsm \leq n$ and $k \leq 3 \omega n$, we have, for any $g \in \Supp(\mu^{*k})$ satisfying \eqref{eq:goodmuk},
\[e^{(\alpha - (2\lambda_1 + 4C)\omega)n} \leq e^{T\lambda_1sm}\norm{\xi \cdot \gamma^{-1}g} \leq e^{(\alpha_0 + 4C\omega)n}.\]
Here we assumed $n$ to be larger than a constant depending on $\gamma$ to beat the implied constant in the $\ll_\gamma$ notation in \eqref{eq:xigammag}. 
With the choice $\omega = \min\{\frac{\alpha}{4\lambda_1 + 8C}, \frac{1}{16CTs}\}$, we can guarantee that this implies
\[e^{\alpha m/2} \leq e^{T\lambda_1sm}\norm{\xi \cdot \gamma^{-1}g} \leq e^m.\]
Putting everything together, we obtain
\[\abs{I_{l_1,\dotsc,l_s,k}(\xi)} \leq e^{-cn}\]
for all $l_1,\dotsc,l_s \in \Lcal$, $k \leq 3 \omega n$ with $l_1 + \dotsb + l_s + k = n$.
This concludes the proof of the theorem.
\end{proof}

\subsection{Proof of Theorems~\ref{thm:aut,base} and \ref{thm:aff,base}}
Let $\mu$ be a Borel probability measure on $\GL_d(\Z) \ltimes \R^d$ having a finite exponential moment.
Let $x \in X$ be a point.
We shall use the shorthand $\nu_n = \mu^{*n} *\delta_x$.
Assume that for some $a \in \Z^d\setminus\{0\}$ and $t \in {(0,\frac{1}{2})}$ and for some large $n$, we have 
\begin{equation}
\label{eq:nunat}
\abs{\widehat{\nu_n}(a)} > t.
\end{equation}

Let $\Gamma$ denote the group generated by $\theta_*\mu$.
Let $G$ denote the Zariski closure of $\Gamma$ and $G^\circ$ the identity component of $G$.
Let $E$ be the subalgebra generated by $G^\circ(\R)$.
Let $\gamma_1,\dots,\gamma_J$ be a complete set of representatives for the cosets in $\Gamma/(\Gamma \cap E)$.
For any integer $m$, the we can decompose
\[(\theta_*\mu)^{*m} = \sum_{j = 1}^J (\gamma_j)_*\mu_{m,j}\]
where $\mu_{m,j}$ is a measure on $\Gamma \cap E$.
By Theorem~\ref{thm:decaymun}, for $m$ large enough, we have the Fourier decay property for each $\mu_{m,j}$,
\[\forall \xi \in E^* \text{ with } e^{\alpha m} \leq e^{\lambda_1  m} \norm{\xi} \leq e^{\alpha_0 n},\quad \abs{\widehat{\mu_{m,j}}(\xi)} \leq e^{-c_0 n}.\]

Writing $\nu_n = \mu^{*m} * \nu_{n-m}$, we have
\begin{align*}
\widehat{\nu_n}(a) & = \iint e(\langle a, g y \rangle) \dd\mu^{*m}(g) \dd \nu_{n-m}(y)\\
& = \sum_{j = 1}^J \iint e(\langle a, g y \rangle) \indic_{\gamma_j E}(\theta(g)) \dd\mu^{*m}(g) \dd \nu_{n-m}(y)
\end{align*}
Thus, \eqref{eq:nunat} implies that there exists $j \in \{1,\dotsc,J\}$ such that
\[t \ll \abse{\iint e(\langle a, g y \rangle) \indic_{\gamma_j E}(\theta(g)) \dd\mu^{*m}(g) \dd \nu_{n-m}(y)}\]
By Hölder's inequality, 
\[t^{2k} \ll \int \abse{\int e(\langle a, g y \rangle) \indic_{\gamma_j E}(\theta(g)) \dd\mu^{*m}(g)}^{2k} \dd \nu_{n-m}(y)\]
After developing the $2k$-power and separating the linear part with the translation part, we obtain
\[t^{2k} \ll \int_{(\gamma_j E)^{2k}} \absbig{\widehat{\nu_{n-m}}((g_1 + \dotsb + g_k - g_{k+1} - \dotsb - g_{2k})^{\tr} a)}  \dd\bigl((\theta_*\mu)^{*m}\bigr)^{\otimes 2k}(g_1,\dotsc,g_{2k}).\]
That is,
\[t^{2k} \ll \int \absbig{\widehat{\nu_{n-m}}(g^{\tr} \gamma_j^{\tr}a)} \dd \bigl(\mu_{m,j}^{\pp k}  \mm \mu_{m,j}^{\pp k}\bigr)(g).\]
Then, the same argument in the proof of \cite[Proposition 4.1]{HS} leads to
\begin{prop}
There are constants $C \geq 1$ and $\sigma > \tau > 0$ depending only on $\theta_*\mu$ such that for $m \geq C\abs{\log t}$, the above implies that there exists a $r_0$-seperated subset $Q \subset \R^d/\Z^d$ such that
\[\nu_{n-m} \Bigl( \bigcup_{x \in Q} B(x,\rho_0)  \Bigr)\geq t^C\]
where $\rho_0 = e^{-\sigma m}\norm{a}$ and $r_0 = e^{\tau m} \rho_0$.
\end{prop}
From here on, the proof of Theorem~\ref{thm:aut,base} is identical to that of \cite[Theorem 1.2]{HS} and that of Theorem~\ref{thm:aff,base} is identical to that of \cite[Theorem 1.3]{HLL}. That is, the Zariski-connectedness condition is not used in the relevant parts in \cite{HS} and \cite{HLL}.

\bibliographystyle{amsplain}
\bibliography{bibautnil}
\end{document}